\documentclass[12pt]{amsart}
\usepackage{amssymb}
\usepackage{verbatim}
\usepackage{graphicx}
\usepackage[cmtip,all]{xy}
\usepackage{amsmath}
\usepackage{amsfonts,amssymb,amsthm}
\usepackage{times}
\usepackage{amscd}
\usepackage{epsfig}
\usepackage{float}
\newtheorem{theo}{Theorem}[section]
\newtheorem{thm}[theo]{Theorem}
\newtheorem{lem}[theo]{Lemma}
\newtheorem{cor}[theo]{Corollary}

\newtheorem{prop}[theo]{Proposition}

\newtheorem*{thm*}{Theorem}

\newtheorem*{thmA}{Theorem A}
\newtheorem*{thmB}{Theorem B}
\newtheorem*{thmC}{Theorem C}
\newtheorem*{thmD}{Theorem D}

\theoremstyle{definition}
\newtheorem{dfn}[theo]{Definition}

\newtheorem*{staa}{Standing Assumption}

\theoremstyle{remark}

\numberwithin{equation}{section}

\newcommand{\ql}{quadrilateral{}}

\newcommand{\lami}{\mathrm{LAM}}

\newcommand{\De}{\Delta}

\newcommand{\F}{\mathcal{F}}
\newcommand{\np}{\mathcal{NP}}

\newcommand{\C}{\mathbb{C}}

\newcommand{\disk}{\mathbb{D}}
\newcommand{\disc}{\mathbb{D}}
\newcommand{\cdisk}{\ol{\mathbb{D}}}

\newcommand{\ol}{\overline}

\newcommand{\sm}{\setminus}
\newcommand{\A}{\mathcal{A}}

\newcommand{\ha}{\hat{a}}
\newcommand{\hb}{\hat{b}}

\newcommand{\m}{\ol{m}}
\newcommand{\n}{\ol{n}}
\newcommand{\hell}{\hat{\ell}}

\newcommand{\qml}{\mathrm{QML}}
\newcommand{\bd}{\mathrm{Bd}}

\newcommand{\lam}{\mathcal{L}}
\newcommand{\hlam}{\widehat{\mathcal{L}}}

\newcommand{\cs}{\mathrm{CS}}

\newcommand{\ch}{\mathrm{CH}}

\newcommand{\si}{\sigma}
\newcommand{\0}{\emptyset}
\newcommand{\uc}{\mathbb{S}}

\newcommand{\g}{\mathfrak{g}}

\newcommand{\Ss}{\mathcal{S}}

\newcommand{\Pp}{\mathrm{CML}}
\newcommand{\pr}{\mathbb L\mathbb{P}}
\newcommand{\prnp}{\pr^{np}}

\newcommand{\qcp}{\mathrm{QCP}}
\newcommand{\fqcp}{\mathcal{QCP}}
\newcommand{\qnp}{\fqcp^{np}}
\newcommand{\rc}{\mathcal{R}}
\newcommand{\Lam}{\mathbb{L}}
\newcommand{\Q}{\Lam\!\mathcal Q}

\begin{document}
\date{January 15, 2015}

\title[The parameter space of cubic laminations]{The parameter space of cubic laminations\\ with a fixed critical leaf}

\author[A.~Blokh]{Alexander~Blokh}

\thanks{The first author was partially
supported by NSF grant DMS--1201450}

\author[L.~Oversteegen]{Lex Oversteegen}


\author[R.~Ptacek]{Ross~Ptacek}

\author[V.~Timorin]{Vladlen~Timorin}

\thanks{The fourth named author was partially supported by
RFBR grants 13-01-12449, 13-01-00969, and AG Laboratory NRU-HSE,
MESRF grant ag. 11.G34.31.0023.}

\address[Alexander~Blokh and Lex~Oversteegen]
{Department of Mathematics\\ University of Alabama at Birmingham\\
Birmingham, AL 35294}

\address[Ross~Ptacek and Vladlen~Timorin]
{Faculty of Mathematics\\
Laboratory of Algebraic Geometry and its Applications\\
Higher School of Economics\\
Vavilova St. 7, 112312 Moscow, Russia }

\address[Vladlen~Timorin]
{Independent University of Moscow\\
Bolshoy Vlasye\-vskiy Pereulok 11, 119002 Moscow, Russia}

\email[Alexander~Blokh]{ablokh@math.uab.edu}
\email[Lex~Oversteegen]{overstee@math.uab.edu}
\email[Ross~Ptacek]{rptacek@uab.edu}
\email[Vladlen~Timorin]{vtimorin@hse.ru}



\begin{abstract}
Thurston parameterized quadratic invariant laminations with a
non-invariant lamination, the quotient of which yields a
combinatorial model for the Mandelbrot set. As a step toward
generalizing this construction to cubic polynomials, we consider
slices of the family of cubic invariant laminations defined by a
fixed critical leaf with non-periodic endpoints. We parameterize
each slice by a lamination just as in the quadratic case, relying on
the techniques of smart criticality
previously developed by the authors.  
\end{abstract}

\maketitle

\section*{Introduction}\label{s:intro}
One of the major goals of complex dynamics is to describe the
structure of the space of complex polynomials of degree $d$ up to
Moebius conjugation. In the quadratic case, this space of
polynomials can be identified with the family $P_c(z)=z^2+c$ with
$c\in \C$ parameterizing the family. Since $c$ is the critical value
of $P_c(z)$ it shows that dynamically motivated points in the plane
(i.e., the critical value of each polynomial) can be used to
parameterize the space of all polynomials.

A major achievement by Thurston \cite{thu85} was his combinatorial
interpretation of the above space. To achieve this, he introduced
several important dynamical and combinatorial notions. One of them
was that of a \emph{lamination}. Laminations are equivalence
relations on the unit circle in the plane with finite (equivalence)
classes whose  convex hulls in the unit disk are pairwise disjoint.
A lamination is \emph{$\si_2$-invariant} if its equivalence classes
are preserved under $\si_2$. Using a Riemann map one can see that
every quadratic polynomial with a locally connected Julia set
corresponds to a unique $\si_2$-invariant lamination.

Below we give a modified (but equivalent) description of Thurston's
results \cite{thu85}. Every $\si_2$-invariant lamination $\sim$ has
a unique \emph{minor} set $m_\sim$ which is the convex hull of the
image of a $\sim$-class of maximal diameter. A striking result by
Thurston is that the minor sets are pairwise disjoint and form a
lamination of the unit circle! The quotient space of the unit circle
under this lamination then parameterizes the set of all
$\si_2$-invariant laminations and serves as a model of the
Mandelbrot set.

In this paper we will consider a slice of the set of all
$\si_3$-laminations. Fix a chord $D=\ol{ab}$ such that
$\si_3(a)=\si_3(b)$ where both $a$ and $b$ are not periodic. We
consider the set $\lami_D$ of all $\si_3$-invariant laminations
$\sim$ such that $a\sim b$ except for laminations of \emph{Siegel
capture type} defined below. For each lamination $\sim$ in $\lami_D$
we define its minor set $m_\sim$ similar to the above. Then we apply
machinery from \cite{bopt14} to show that, similar to the case of
the Mandelbrot set, the collection of all such minor sets is a
lamination itself so that the quotient space of the unit circle by
this lamination parameterizes $\lami_D$.

We now describe how the paper is organized. While the precise definitions
are given later, 
here we give heuristic versions of some of them. Also, even though
all notions can be introduced for any degree $d$, our main focus
here is cubic, hence some definitions are given only in the cubic
case. 

In Section~\ref{s:prelims} we provide the motivating background from
complex dynamics, define laminations, geolaminations (which will
always be denoted by $\lam$), leaves, (infinite, all-critical) gaps,
and related notions. In particular (cf. \cite{bmov13}), we define
\emph{proper geo\-la\-mi\-na\-tions} as geo\-la\-mi\-na\-tions such
that no leaf can have one periodic and one non-periodic endpoint. We
also review the relevant combinatorial machinery developed in
\cite{bopt14}. Then we define less standard notions, some of which
were introduced in \cite{bopt14}. We give an overview of these
notions here. In the rest of the Introduction we will mostly
consider the cubic case.

A set is a \emph{generalized (critical) \ql{}} if it is a critical
leaf, an all-critical triangle (which collapses to a point), a
\emph{collapsing \ql{}} (i.e., a \ql{} which maps to a chord), or an
all-critical \ql{} (in the cubic case the last case is impossible).
If $\lam$ is cubic, a \emph{quadratically critical portrait
(qc-portrait) of $\lam$} is an ordered pair of \emph{distinct}
generalized \ql s that are leaves or gaps of $\lam$. In fact, cubic
qc-portraits can be defined without geolaminations: an
\emph{admissible (cubic) qc-portrait} is an ordered pair of
generalized \ql s such that (1) a generalized \ql{} with a periodic
vertex either has a degenerate image or, otherwise, has a periodic
edge, and (2) they and all their images intersect at most over a
common edge or vertex. By a standard pullback construction an
admissible cubic qc-portrait can be included in a cubic
geolamination. Here we study only the family $\qnp_3$ of admissible
qc-portraits $\qcp$ such that the second set in $\qcp$ is a critical
leaf $D$ with non-periodic endpoints.

In general, a geolamination may have no qc-portrait; it has a
qc-portrait if and only if all its critical sets are collapsing \ql
s, critical leaves or all-critical gaps \cite{bopt14}. However,
invariant geolaminations can be \emph{tuned} to acquire a
qc-portrait without compromising the dynamics. To this end,
generalized \ql s (which form a qc-portrait) are inserted into
critical gaps of a given proper cubic geolamination $\lam$. Using a
well-known procedure, one can complete the inserted leaves and their
images with their pullbacks, and in this way define a new tuned
invariant geolamination. Observe that the new tuned geolamination is
not necessarily proper as we may need to insert new leaves which
connect periodic and non-periodic points.

Note, that here we consider proper geolaminations $\lam$ which have
a critical leaf $D$ with non-periodic endpoints. For such
geolaminations there are a few possibilities concerning their
critical sets (see Subsection~\ref{ss:pr_geolam}). First, $\lam$ can
have a finite first critical set $C$. By properties of invariant
geolaminations, $C$ is a gap or a leaf on which $\si_3$ acts
two-to-one (unless $D$ is an edge of $C$ and so the point $\si_3(D)$
has all three of its preimages in $C$). Thus, if $C$ is finite,
there are two cases. First, $C$ can be a $2n+1$-gon with $D$ being
one of its edges such that one can break down all its edges into
pairs of ``sibling edges'', i.e. leaves with the same image except
for $D$ (one can say that the ``sibling edge'' of $D$ is the vertex
of $C$ not belonging to $D$ and with the same image as $D$). Second,
$C$ can be a $2m$-gon such that $D$ is not an edge of $C$ in which
case $\si_3|_C$ is two-to-one.

Also, $C$ could be an infinite gap. Then there are two cases. First,
$C$ may be a periodic Fatou gap of some period $k$ and degree $2$
(observe that in this case $D$ may well be an edge of $C$). Second,
there may exist a periodic Siegel gap $U$ with $D$ being one of its
edges and an infinite gap $V$ such that $\si_3|_{\bd(V)}$ is
two-to-one and $V$ eventually maps to $U$. We call such
(geo)laminations \emph{Siegel (geo)laminations of capture type}.

Now, let $\prnp_3$ be the family of all cubic proper geolaminations
which have a critical leaf with non-periodic endpoints \emph{except
for Siegel geolaminations of capture type}. If $\lam\in \prnp_3$ has
a critical leaf $D$ with non-periodic endpoints, then a qc-portrait
$\qcp=(Q, D)$ is called \emph{privileged for $\lam$} if and only if
$Q\subset C$ where $C\ne D$ is a critical set of $\lam$ and either
$C$ is finite, or $C$ is a periodic Fatou gap of degree two and
period $k$ and $Q$ is a collapsing \ql{} which is the convex hull of
a (possibly degenerate) edge $\ell$ of $C$ of period $k$ and another
edge $\hell$ of $C$ such that $\si_3(\ell)=\si_3(\hell)$.

In Section~\ref{s:results} we state our main results. We show that
for each $\lam\in\prnp_3$ there are only finitely many privileged
qc-portraits and for every admissible qc-portrait $\qcp\in\qnp_3$,
there exists $\lam_\qcp\in\prnp_3$ such that $\qcp$ is privileged
for $\lam_\qcp$ and the lamination induced by $\lam_\qcp$ is unique.
For a critical leaf $D$ with non-periodic endpoints, $\Ss_D$ denotes
the collection of all admissible qc-portraits $(Q, D)$ of $\qnp_3$
with $D$ as the second element. To each such qc-portrait $(Q, D)$ we
associate the set (a chord or a point called \emph{minor})
$\si_3(Q)\subset \cdisk$. For each such chord we identify its
endpoints, extend this identification by transitivity and define the
corresponding equivalence relation $\sim_D$ on $\uc$. Our main
result is that $\sim_D$ is itself a lamination whose quotient is a
parameterization of $\lami_D$.

Section~\ref{s:priv_port} describes how we will tag
$\si_3$-invariant geolaminations with privileged qc-portraits and
shows that the tagging generalizes Thurston's tagging of
$\si_2$-invariant geolaminations by minors. We show that any
qc-portrait which contains a critical leaf without periodic
endpoints is privileged for some geolamination.  This implies that
when we fix a critical leaf with non-periodic endpoints we obtain a
closed set of qc-portraits.

In Section~\ref{s:interminor} we study possible intersections of
minors of our qc-portraits. Finally, in Section~\ref{s:unlink} we
construct the parameter lamination $\sim_D$.






\section{Preliminaries}\label{s:prelims}

We assume basic knowledge of complex dynamics and use standard
notation. Let $P:\C\to\C$ be a polynomial of degree $d\ge 2$,
$A_\infty$ be the basin of attraction of infinity, and
$J_P=\bd(A_\infty)$ be the Julia set of $P$.  When $J_P$ is
connected, $A_\infty$ is simply connected and conformally isomorphic
to $\C\sm\disc$ by a unique isomorphism $\phi:\C\sm\ol\disc\to
A_\infty$ asymptotic to the identity at $\infty$.  By a theorem of
B\"ottcher (see, e.g., \cite[Theorem 9.1]{mil00}), $\phi$ conjugates
$P|_{A_\infty}$ with $z^d|_{\C\sm\ol\disc}$. If $J_P$ is locally
connected, then $\phi$ extends continuously to a semiconjugacy
$\ol\phi$ between $\si_d=z\mapsto z^d|_\uc$ and $P|_{J_P}$:

\begin{equation}\label{e:lam_conj}\begin{CD}
\uc @>{\si_d|_{\uc}}>> \uc \\
@V{\ol\phi}VV @V{\ol\phi}VV \\
J_P @>{P|_{J_P}}>> J_P \\
\end{CD}\end{equation}

The \textit{lamination generated by} $P$ is the equivalence relation
$\sim_P$ on $\uc$ whose classes are \emph{$\ol\phi$-fibers}, i.e. point-preimages
under $\ol\phi$.  We call
$\widehat J_P=\uc/\sim_P$ a \textit{topological Julia set} and
$\widehat P$, the corresponding map induced on $\widehat J_P$ by $\si_d$, a
\textit{topological polynomial}.  Evidently $P|_{J_P}$ and
$\widehat P|_{\widehat{J_P}}$ are conjugate.  The collection $\lam_P$
of chords of $\ol\disc$ which are edges of convex hulls of
$\sim_P$ classes is called the \textit{geolamination
generated by }$P$.

\subsection{Laminations, Geolaminations, and their Properties}

Laminations and geolaminations can be defined independently of
polynomials.  This approach is due to Thurston
\cite{thu85} who constructed a combinatorial model of the
Mandelbrot set by parameterizing quadratic geolaminations.
We define laminations to be equivalence relations
without underlying dynamics.

\begin{dfn}[Laminations]\label{d:lam}
An equivalence relation $\sim$ on the unit circle $\uc$ is called a
\emph{lamination} if either $\uc$ is one $\sim$-class (such
laminations are called \emph{degenerate}),
or the following holds:

\noindent (E1) the graph of $\sim$ is a closed subset of $\uc \times
\uc$;

\noindent (E2) the convex hulls of distinct equivalence classes are
disjoint;

\noindent (E3) each equivalence class of $\sim$ is finite.
\end{dfn}

These properties are satisfied by laminations generated by
polynomials.  Property (E1) is necessary for $\uc/\sim$ to be a
metric space, (E2) follows from the semiconjugacy $\ol\phi$ (a proof
can be found in \cite[Proposition II.3.3]{sch09}), and (E3) follows
from \cite[Lemma 18.12]{mil00} for (pre)periodic points and from
Kiwi (\cite{kiw02}) for points with infinite orbit.
Definition~\ref{d:si-inv-lam} enforces dynamics on laminations
consistent with those generated by polynomials.

\begin{dfn}[Laminations and dynamics]\label{d:si-inv-lam}
An equivalence relation $\sim$ is called ($\si_d$-){\em invariant}
if:

\noindent (D1) $\sim$ is {\em forward invariant}: for a $\sim$-class
$\g$, the set $\si_d(\g)$ is a $\sim$-class;

\noindent (D2) for any $\sim$-class $\g$, the map $\si_d: \g\to
\si_d(\g)$ extends to $\uc$ as an orientation preserving covering
map such that $\g$ is the full preimage of $\si_d(\g)$ under this
covering map.
\end{dfn}

There is a useful geometric object associated with each lamination.

\begin{dfn}\label{d:geolam}
For a lamination $\sim$, take convex hulls of all its classes and
consider the family of chords which are their edges (together with
all single points of the unit circle). This family of (possibly
degenerate) chords is called a \emph{geolamination generated by
$\sim$} and is denoted by $\lam_\sim$.
\end{dfn}

Definition~\ref{d:geolam} makes sense for non-invariant and
invariant laminations; for invariant laminations we use notation
described in Definition~\ref{d:qlam}.

\begin{dfn}\label{d:qlam}
The family of all invariant geolaminations generated by laminations
is denoted by $\Q$; if we consider only $\si_d$-invariant
laminations, the corresponding family of geolaminations is denoted
by $\Q_d$. Geolaminations from $\Q$ ($\Q_d$) are called
\emph{$\Q$-geolaminations ($\Q_d$-geolaminations)}.
\end{dfn}

Similar notions were defined by Thurston in \cite{thu85} with no
equivalence relations involved. For a collection $\rc$ of chords of
$\disk$ set $\rc^+=\bigcup\rc$.

\begin{dfn}[Geolaminations]\label{d:geolam1}
Two chords of $\cdisk$ are called \emph{unlinked} if they are
disjoint in $\disk$. A \emph{geolamination} is a collection $\lam$
of (perhaps degenerate) pairwise unlinked chords of $\cdisk$ called
\emph{leaves} such that $\lam^+$ is closed, and all points of $\uc$
are elements of $\lam$. A \textit{gap} of $\lam$ is the closure of a
component of $\disk\sm\lam^+$. If $\lam=\lam_\sim$ is a
geolamination generated by a lamination $\sim$, gaps of $\lam_\sim$
will also be called \emph{gaps of $\sim$}. A gap $G$ is \emph{finite
(infinite)} depending on whether $G\cap\uc$ is finite (infinite).
Gaps $G$ such that $G\cap\uc$ is uncountable, are called \emph{Fatou
gaps}. For Fatou gaps $G$ define the monotone map $\psi_G:\bd(G)\to
\uc$ which collapses all edges of $G$ to points.
\end{dfn}

If $x, y\in\uc$ are the endpoints of a chord $\ell$, set
$\ell=\ol{xy}$. Given a closed set $A\subset \uc$, let
$\si_d(\ch(A))=\ch(\si_d(A))$ (in particular,
$\si_d(\ol{xy})=\ol{\si_d(x)\si_d(y)}$). Any non-degenerate chord
$\ell$ such that $\si_d(\ell)$ is a point is called \emph{critical}.
Given a geolamination $\lam$, we linearly extend $\si_d$ over leaves
of $\lam$; clearly, this extension is continuous and well-defined.
With this extension, we can define \emph{$\si_d$-invariant
geolaminations}. Here we rely on \cite{bmov13} where the definition
is a bit different from Thurston's \cite{thu85}.

\begin{dfn}[Invariant sibling geolaminations \cite{bmov13}]\label{d:sibli}
A geolamination $\lam$ is \emph{($\si_d$-)invariant (sibling)
geolamination} provided:

\begin{enumerate}

\item \label{d:sibli:fi} for each $\ell\in\lam$, we have $\si_d(\ell)\in\lam$,

\item \label{d:sibli:bi}for each $\ell\in\lam$ there exists $\ell^*\in\lam$
    so that $\si_d(\ell^*)=\ell$.

\item \label{d:sibli:si} for each $\ell\in\lam$ such that $\si_d(\ell)$
    is non-degenerate, there exist $d$ \textbf{pairwise disjoint}
    leaves $\ell_1$, $\dots$, $\ell_d$ in $\lam$ so that
    $\ell=\ell_1$ and $\si_d(\ell_i)=\si_d(\ell)$ for all
    $i=1$, $\dots$, $d$.

\end{enumerate}

\noindent If a geolamination satisfies only condition
\ref{d:sibli:fi}, 
we call it a \emph{forward
($\si_d$-)invariant sibling geolamination}.
\end{dfn}

\emph{Any} leaf $\ell^*$ with $\si_d(\ell^*)=\ell$ is called a
\textit{pullback} of $\ell$. Also, two leaves with the same
$\si_d$-image are called \textit{siblings}, and the leaves
$\ell_1=\ell,\dots,\ell_d$ in (\ref{d:sibli:si}) are said to form a
\textit{disjoint sibling collection of }$\ell$ \cite{bmov13}. The
space of all $\si_d$-invariant sibling geolaminations is denoted
$\Lam_d$.

Thurston defines $\si_d$-invariant geolaminations differently. He
requires that geolaminations satisfy (1) above, be such that each
leaf $\ell$ has $d$ \emph{disjoint} pullbacks, and be
\emph{gap invariant}. 

\begin{dfn}[Gap invariance \cite{thu85}]\label{d:gap-inv}
Suppose that for any gap $G$ of a geolamination $\lam$ the set
$H=\si_d(G)$ is either a point of $\uc$, a leaf of $\lam$, or a gap
of $\lam$. Moreover, if $H$ is a gap then $\si_d:\bd(G)\to\bd(H)$ is
the positively oriented composition of a monotone map and a covering
map. Then $\lam$ is said to be \emph{gap invariant}.
\end{dfn}

Let us now define the degree of $\si_d$ on a gap or leaf of $\lam$.

\begin{dfn}[Degree of $\si_d$ on a gap or leaf]\label{d:degreeg}
Let $G$ be a gap or leaf of a gap invariant geolamination and let
$H=\si_d(G)$. Then the \emph{degree of $\si_d|_G$} is either the
cardinality of $G\cap \uc$ (if $H$ is a point), or, otherwise, the
number of components of the set $\si_d^{-1}(x)\cap \bd(G)$ where $x$
is any point of $\bd(H)$.
\end{dfn}

It is easy to see that the degree of $\si_d|_G$ does not depend on
the choice of the point $x\in \bd(H)$. In case when $\si_d(G)$ is a
gap, the degree of $\si_d|_G$ can be defined as the degree of a
covering map involved in the representation of $\si_d|_{\bd(G)}$ as
the composition of a monotone map and a covering map.

\begin{dfn}\label{d:crigap}
A gap $G$ all of whose edges are critical is said to be
\emph{all-critical}. A gap $G$ is called \emph{critical} if the
degree of $\si_d|_G$ is greater than $1$.
\end{dfn}


Definition~\ref{d:sibli} conveniently deals only with leaves. This
leads to useful results on properties of $\si_d$-invariant sibling
geolaminations \cite{bmov13}. 
We will agree to endow any family of compact subsets of $\cdisk$
with the Hausdorff metric and the induced topology. In particular,
we define a natural topology on $\Lam_d$ induced by the Hausdorff
metric on subcontinua $\lam^+$ of $\cdisk$ where $\lam\in \Lam_d$
(from now on talking about $\Lam_d$ we mean the just defined
topological space rather than a family of geolaminations).

\begin{thm}[\cite{bmov13}]\label{t:summary}
The following results hold:

\begin{enumerate}

\item $\Q_d\subset \Lam_d$ (i.e., geolaminations generated by
$\si_d$-invariant  laminations are
$\si_d$-invariant sibling geolaminations);

\item $\Lam_d$ is compact (thus, $\Lam_d$ contains
all limits of geolaminations from $\Q_d$);

\item $\si_d$-invariant sibling geolaminations are
$\si_d$-invariant in the sense of Thurston.

\end{enumerate}

\end{thm}

We will also need a few technical results of \cite{bmov13}. Denote by
$<$ the \emph{positive (circular) order on $\uc$}.  That is, given
points $x, y, z\in\uc$, $x<y<z$ if the counterclockwise arc of $\uc$
from $x$ to $y$ contains $y$ in its interior. A consequence of gap
invariance is that if an arc in the boundary of a gap of a
$\si_d$-invariant geolamination is mapped injectively by $\si_d$, then
$\si_d$ preserves the (circular) order of points of that arc.  In
\cite{bmov13} this property was explored for leaves of
$\si_d$-invariant sibling geolaminations emanating from the same point
of $\uc$ which may or may not be edges of a common gap.

\begin{lem}[\cite{bmov13}, Lemma 3.7]\label{l:triod_ord} Let $\lam$ be a
$\si_d$-invariant sibling geolamination and $T\subset\lam^+$ be an arc
consisting of two leaves with a common endpoint or a triod consisting
of three leaves with a common endpoint.  Suppose that $Y\subset\lam^+$
is an arc (triod) such that $\si_d(Y)=T$ and $\si_d|_Y$ is one-to-one.
Then $\si_d|_{Y\cap\uc}$ preserves circular order.
\end{lem}

Corollary~\ref{c:nocrit} easily follows from
Lemma~\ref{l:triod_ord}.

\begin{cor}\label{c:nocrit}
Let $\lam$ be a $\si_d$-invariant sibling geolamination. Let
$x<y<z<x$ be the endpoints of two non-critical leaves $\ol{xy}$ and
$\ol{yz}$ of $\lam$ such that $\si_d(x)=\si_d(z)$. If now $z<t<x$
and $\ol{yt}$ is a leaf of $\lam$ then either $\si_d(t)=\si_d(y)$ or
$\si_d(t)=\si_d(x)$; in particular, there are at most $2d-3$ such
leaves $\ol{yt}$.
\end{cor}

\begin{proof}
Suppose otherwise. Then if we apply Lemma~\ref{l:triod_ord} to
$\ol{xy}\cup \ol{yt}$ we see that
$\si_d(x)<\si_d(y)<\si_d(t)<\si_d(x)$. On the other hand, if we
apply Lemma~\ref{l:triod_ord} to $\ol{yz}\cup \ol{yt}$ we see that
$\si_d(y)<\si_d(z)<\si_d(t)<\si_d(y)$. Since $\si_d(x)=\si_d(z)$,
this is a contradiction.
\end{proof}

Corollary~\ref{c:finsib} follows from Corollary~\ref{c:nocrit}. By
definition, if a non-critical leaf $\ell$ of a $\si_d$-invariant
sibling geolamination has overall $d-1$ other leaves of $\lam$ which
map to $\si_d(\ell)$, then all these $d$ leaves are pairwise
disjoint and form a unique disjoint sibling collection of leaves
with image $\si_d(\ell)$.

\begin{cor}\label{c:finsib}
There are finitely many leaves $\ell$ such that more than $d$ leaves
have the image $\si_d(\ell)$.
\end{cor}

\begin{proof}
We may assume that $\ell$ is non-critical.  If $\ell$ has more than
$d-1$ other leaves with the image $\si_d(\ell)$ then at least two of
them will share an endpoint (say, $y$) with the two other endpoints
(say, $x$ and $z$) mapped to the same point. Set $\ch(x, y,
z)=T_\ell$. Call $\ol{xz}$ the \emph{base} of $T_\ell$ and call $y$
the \emph{top vertex} of $T_\ell$. It is easy to see that there are
at most finitely many (in fact, no more than $2d-3$) triangles whose
bases are pairwise unlinked; let $T^1, \dots, T^k$ be a maximal
collection of such triangles with pairwise unlinked bases. Any
remaining triangle $T_\ell$ is such that its base crosses the base
of some triangle $T^j, 1\le j\le k$. Since sides of all these
triangles which are not their bases are \emph{unlinked} leaves, then
$T_\ell$ and $T^j$ must share a vertex. Clearly, the number of
vertices of triangles $T^1, \dots, T^k$ is finite. By
Corollary~\ref{c:nocrit}, each of them can be the top vertex of
finitely many triangles $T_\ell$. Hence the overall number of such
triangles is finite as desired.
\end{proof}

This helps in studying disjoint sibling collections.

\begin{lem}\label{l:nonisol}
If $\ell$ is a non-critical non-isolated leaf, then there exists a
disjoint sibling collection which contains $\ell$ and consists of
non-isolated leaves. I.e., if $\hell$ is a non-degenerate
non-isolated from one (two) sides leaf then there is a disjoint
sibling collection of leaves mapped to $\hell$ such that leaves in
this collection are non-isolated from the appropriate side(s).
\end{lem}

\begin{proof}
Let $\ell^1_i\to \ell^1=\ell$ be a sequence of leaves which converge
to $\ell$ from one side. By Corollary~\ref{c:finsib} we may assume
that for each $i$ the leaf $\ell_i$ is a member of a disjoint
sibling collection $L_i=\{\ell_i^1, \ell^2_i, \dots, \ell^d_i\}$
which includes \emph{all} leaves mapped to $\si_d(\ell^1_i)$. By
compactness we may assume that leaves in these collections are
ordered so that for each $1\le j\le d$, $\lim_i \ell^j_i=\ell^j$.
Clearly, $\si_d(\ell^j)=\si_d(\ell)$ and all leaves $\ell^j$ are
non-isolated.

To show that all these leaves are pairwise disjoint, observe that
since the leaves $\ell^1_i$ approach $\ell^1$ from one side, then
for any $j$, the leaves $\ell^j_i$ approach $\ell^j$ from one side.
Thus, the leaves $\si_d(\ell^j_i)$ approach $\si_d(\ell^j)$ from one
side too. If now, say, $\ell^2=\ell^3$ then it would follow that
leaves $\ell^2_i, \ell^3_i$ are disjoint, close to each other, and
have the same image, which is impossible. Suppose that $\ell^2\ne
\ell^3$ come out of the same point but are otherwise disjoint. Then
the leaves $\ell^2$ and $\ell^3$ form a wedge, and the leaves
$\ell^2_i, \ell^3_i$ are located outside the wedge close to the
appropriate sides of the wedge (recall that all these leaves come
from the same geolamination and are therefore pairwise unlinked).
Then their images approach $\si_d(\ell^1)$ from opposite sides, a
contradiction.
\end{proof}

Lemma~\ref{l:nonisol} allows us to apply a certain ``cleaning''
procedure to construct other $\si_d$-invariant sibling
geolaminations out of a given geolamination.

\begin{thm}\label{t:cleanup}
Let $\lam$ be a $\si_d$-invariant sibling geolamination and $\A$ be
a family of grand orbits of its leaves. Construct a collection of
leaves $\lam'$ by removing from $\lam$ all isolated leaves belonging
to $\A$. Then $\lam'$ is a $\si_d$-invariant sibling geolamination.
\end{thm}

\begin{proof}
Since non-isolated leaves map to non-isolated leaves then $\lam'$ is
forward invariant. Let $\ell'\in \lam'$ be non-degenerate. Then
either $\ell'\notin \A$, or $\ell'\in \A$ is non-isolated. In the
former case by definition there is a leaf $\ell''\in \lam$ such that
$\si_d(\ell'')=\ell'$, and since $\ell'\notin \A$ it follows that
$\ell''\notin \A$ and hence $\ell''\in \lam'$. In the latter case
the claim follows from Lemma~\ref{l:nonisol}. Thus, any
non-degenerate leaf from $\lam'$ has a preimage in $\lam'$. Finally
let $\ell'\in \lam'$ be non-critical. Then again either $\ell'\notin
\A$, or $\ell'\in \A$ is non-isolated. In the former case the claim
is obvious as all disjoint sibling collections of $\ell'$ are
disjoint from $\A$. In the latter case the claim follows from
Lemma~\ref{l:nonisol}.
\end{proof}

Arguments, similar to those in the proof of Corollary~\ref{c:nocrit}
and based upon Lemma~\ref{l:triod_ord}, prove
Corollary~\ref{c:sameperiod}; we leave the proof to the reader.

\begin{cor}\label{c:sameperiod}
If $\lam$ is a $\si_d$-invariant sibling geolamination and $\ol{ab}$
is a leaf of $\lam$ with periodic endpoints then $a$ and $b$ are
periodic of the same period.
\end{cor}


The paper \cite{bmov13} also explored how disjoint sibling
collections must be located in $\ol\disc$.  To state the
corresponding lemma we need a notational agreement. If $X$ is a
collection of 
pullbacks of a leaf $\ol{ab}$, we denote the endpoints of chords of
$X$ by the same letters as for the endpoints of their images but
with a hat and distinct subscripts, and call them correspondingly
($a$-points, $b$-points etc). Lemma~\ref{l:sibling_loc}  is
typically used in the situation of Figure~\ref{f:sibs}, when no
leaves cross the critical chord joining two sibling leaves. We say
that two distinct chords \emph{cross} each other, or are
\emph{linked}, if they intersect inside the open unit disk $\disk$.


\begin{figure}[h!]
\centering\def\svgwidth{.61\columnwidth}\fbox{\fbox{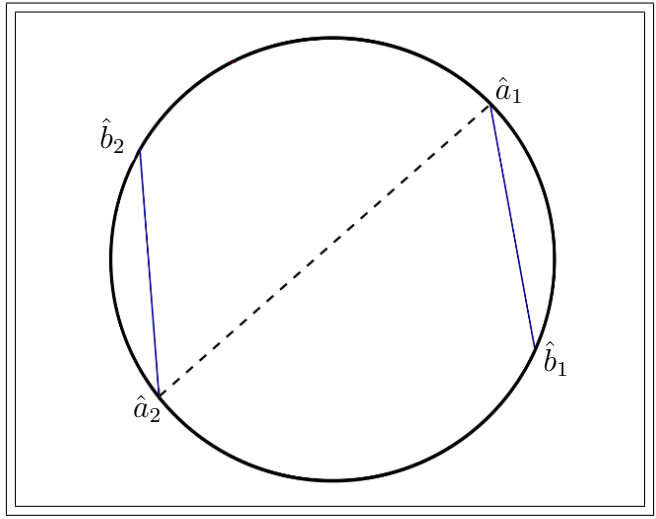}}
\caption{Siblings and critical leaves.} Siblings must be on opposite
sides of the chord $\ol{\ha_1\ha_2}$ which is not crossed by leaves
of the disjoint sibling collection. \label{f:sibs}
\end{figure}


\begin{lem}[\cite{bmov13}, Lemma 3.8]\label{l:sibling_loc}
Let $X$ be a disjoint sibling collection of leaves mapped to a leaf
$\ol{ab}$ and $\ol{\ha_1\hb_1}$, $\ol{\ha_2\hb_2}$ be two leaves
from $X$. Then the number of leaves from $X$ crossing the chord
$\ol{\ha_1\ha_2}$ inside $\disk$ is even if and only if either
$\ha_1<\hb_1<\ha_2<\hb_2$ or $\ha_1<\hb_2<\ha_2<\hb_1$. In
particular, if there exists a concatenation $Q$ of chords connecting
$\ha_1$ and $\ha_2$, disjoint  with leaves of $X$ except the points
$\ha_1, \ha_2$, then either $\ha_1<\hb_1<\ha_2<\hb_2$ or
$\ha_1<\hb_2<\ha_2<\hb_1$.
\end{lem}



We will also need the following easy corollary of
Lemma~\ref{l:sibling_loc}.

\begin{cor}\label{c:sibling_loc}
Let $X$ be a disjoint sibling collection of leaves mapped to a leaf
$\ol{ab}$ and $\ol{\ha_1\hb_1}$, $\ol{\ha_2\hb_2}$ be two leaves
from $X$. Suppose that there exists a concatenation $Q$ of pairwise
non-crossing critical chords connecting $\ha_1$ and $\ha_2$, with
endpoints separated by $\ol{\ha_1 \ha_2}$ from $\hb_1$ inside
$\disk$, and disjoint from the leaves of $X$ except for their
endpoints. Then either $\ha_1<\hb_1<\ha_2<\hb_2$ or
$\ha_1<\hb_2<\ha_2<\hb_1$.
\end{cor}

\begin{proof}
Without loss of generality, assume that $\ha_1<\ha_2<\hb_1$. By the
assumptions of the lemma there is a chain $Y$ of pairwise non-crossing
critical chords $\ol{\ha_1x_1}, \ol{x_1x_2}, \dots, \ol{x_n\ha_2}$ with
$\ha_1<x_1<x_2<\dots<x_n<\ha_2$ each of which does not cross leaves
from $X$. By definition, $X$ must contain $n$ disjoint leaves with
endpoints $x_1, \dots, x_n$ respectively. Moreover, applying
Lemma~\ref{l:sibling_loc} step by step and using the fact that chords
in $Y$ do not cross chords of $X$ we see that $X$ contains leaves
$\ol{y_1x_1}, \dots, \ol{y_nx_n}$ such that
$${\ha_1<y_1<x_1<y_2<x_2<\dots<y_n<x_n<\hb_2<\ha_2}$$ which implies
that $\ha_1<\hb_2<\ha_2<\hb_1$ as desired.
\end{proof}

One benefit of working with geolaminations is the ability to construct
$\si_d$-invariant geolaminations from a small generating set.  This is
due to Thurston \cite{thu85} for his definition of $\si_d$-invariance,
but is easily modified for $\si_d$-invariant sibling geolaminations.
The geolamination of Theorem~\ref{t:pullback_lam} is constructed by
iteratively adding preimages of leaves of $\widehat\lam$ and taking the
closure after countably many steps.

\begin{thm}[cf \cite{thu85}]\label{t:pullback_lam}
Let $\widehat\lam$ be a forward $\si_d$-invariant sibling
geolamination. Then there exists a $\si_d$-invariant geolamination
$\lam\supset\widehat\lam$ for which grand orbits of leaves of $\widehat\lam$
are dense.
\end{thm}

In what follows we call $\si_2$-invariant ($\si_3$-invariant) sibling
geolaminations respectively \emph{quadratic} and \emph{cubic} (sibling)
geolaminations.

\subsection{Accordions and Linked Geolaminations}\label{ss:accord}
By \cite{thu85}, a quadratic geolamination $\lam$ has one or two
longest leaves called \textit{majors} of $\lam$. If the major is
unique then it must be a diameter, otherwise the two majors are
disjoint. In any case, the majors are denoted $M_\lam$ and
$M'_\lam$. The \textit{minor} $m_\lam$ of $\lam$ is defined as
$m_\lam=\si_2(M_\lam)=\si_2(M'_\lam)$. Thurston proves that no two
minors intersect in $\disc$ and the collection $\qml$ of all
quadratic minors is a geolamination itself called the
\emph{Quadratic Minor Lamination}.

In \cite{bopt14}, we partially generalize this result to cubic
geolaminations.  Using minors alone for such geolaminations is
impossible.  While in the quadratic case $m_\lam$ is uniquely pulled
back under $\si_2$ to a pair of majors, there are multiple ways to pull
leaves back under $\si_3$ which result in very different
geolaminations.  To avoid this ambiguity, we associate cubic
geolaminations with ordered pairs of \emph{generalized critical
quadrilaterals}; such pairs are called \emph{qc-portraits}.  In the
case of $\si_2$, this is equivalent to associating to a quadratic
geolamination $\lam$ the set $\ch(\si_2^{-1}(m_\lam))$. 

\begin{dfn}[Generalized critical \ql s]\label{d:qll}
A \emph{generalized quadrilateral} $Q$  is the circularly ordered
4-tuple $[a_0,a_1,a_2,a_3]$ of marked points $a_0\le a_1\le a_2\le
a_3\le a_0$ in $\uc$. Such $Q$ is said to be
\emph{($\si_d$-)critical} if $\ol{a_0a_2}$ and $\ol{a_1a_3}$ are
non-degenerate $\si_d$-critical chords (called \emph{spikes}).
Generalized $\si_d$-critical quadrilaterals $[a_0,a_1,a_2,a_3]$,
$[a_1,a_2,a_3,a_0]$, and the other circular permutations of the
vertices, are viewed as equal. The chords $\ol{a_0a_1}, \ol{a_1a_2},
\ol{a_2a_3}, \ol{a_3a_0}$ are called \emph{edges} of $Q$.
\end{dfn}

Note that a generalized critical \ql{} $Q$ is either a critical
leaf, or an all-critical triangle, or an all-critical \ql
(impossible if the degree is three), or a \ql{} which collapses to a
leaf (a \emph{collapsing \ql}). Although the definition of a
qc-portrait is suitable for $\si_d$-invariant sibling
geolaminations, we simplify it here for $\si_3$. When saying that a
geolamination $\lam$ \emph{has (or is with) a generalized critical
\ql{} $Q$} we mean that $Q$ is a leaf or a gap of $\lam$.

\begin{dfn}[Quadratic criticality]\label{d:quacintro}
Let $(\lam, \qcp)$ be a cubic geolamination with an ordered pair
$\qcp$ of {\em distinct} generalized critical \ql s which do not
intersect in $\disc$ other than over a common edge/vertex. Then
$\qcp$ is called a \emph{quadratically critical portrait
$($qc-portrait$)$ for $\lam$} while the pair $(\lam, \qcp)$ is
called a \emph{geolamination with qc-portrait}.
\end{dfn}

Note that we do \emph{not} require that the two generalized critical
\ql s be disjoint. In fact, one may even be a subset of the other.
However we do require that they be distinct. Thus, a critical leaf
repeated twice is \emph{not} a qc-portrait.

In the quadratic case the notion of a qc-portrait reduces to that of
a critical leaf or collapsing \ql (recall, that a collapsing \ql{}
is a critical \ql{} whose image is a non-degenerate chord). In this
case the image of a collapsing \ql{} or a critical leaf of a
quadratic geolamination $\lam$ is the minor $m_\lam$ of $\lam$.

Now we discuss \emph{strong linkage} between generalized critical
\ql s and then \emph{linkage} between cubic geolaminations with
qc-portraits. For a generalized \ql{} $Q$, call a component of
$\uc\sm Q$ a \emph{hole} of $Q$.

\begin{dfn}[Strongly linked \ql s]\label{d:strolin}
Let $A$ and $B$ be generalized \ql s. Say that $A$ and $B$ are
\emph{strongly linked} if the vertices of $A$ and $B$ can be numbered so
that
$$a_0\le b_0\le a_1\le b_1\le a_2\le b_2\le a_3\le b_3\le a_0$$
where $a_i$, $0\le i\le 3$, are vertices of $A$ and $b_i$, $0\le
i\le 3$ are vertices of $B$. Note that the vertices of the two \ql s
here need not strictly alternate on the circle.
\end{dfn}

Definition~\ref{d:qclink1} is a cubic version of the corresponding
one from \cite{bopt14}.

\begin{dfn}[Linked qc-portraits]\label{d:qclink1}
Let qc-portraits $\qcp_1=(C^1_1,C^2_1)$ and
$\qcp_2=(C^1_2,C^2_2)$ be such that for every $1\le i\le 2$ the
sets $C^i_1$ and $C^i_2$ are either strongly linked generalized
critical \ql s or share a spike. Then the qc-portraits are said to be
\emph{linked} while the critical sets $C^i_1$ and $C_2^i$, $i=1, 2$
are called \emph{associated}. If $(\lam_1, \qcp_1)$ and $(\lam_2,
\qcp_2)$ are geolaminations with qc-portraits then they are called
\emph{linked} if either $\qcp_1$ and $\qcp_2$ are linked, or
$\lam_1$ and $\lam_2$ share an all-critical triangle.
\end{dfn}

Observe that in \cite{bopt14} geolaminations $\lam_1, \lam_2$ from
Definition~\ref{d:qclink1}, for which either both pairs of
associated generalized critical \ql s in qc-portraits share a spike
or $\lam_1, \lam_2$ themselves share an all-critical triangle, are
said to be \emph{essentially equal} rather than linked. However for
our purposes this fine distinction does not matter much, thus in
this paper we use a little less precise definition above.

Thurston's result that no quadratic geolaminations have linked minors
is equivalent to the statement that there are no quadratic
geolaminations with linked qc-portraits.
Also, in \cite{bopt14} the notion of linked geolaminations was
introduced and various results were obtained in the degree $d$ case;
the results stated in the rest of this subsection are restatements
of the results of \cite{bopt14} for the cubic case. The main tool
used in \cite{bopt14} to study linked $\si_d$-invariant sibling
geolaminations with qc-portraits is an \textit{accordion} designed
to track linked leaves from linked geolaminations.

\begin{dfn}\label{d:accord}
Let $\ell_1,\ell_2$ be leaves of $\si_d$-invariant sibling
geolaminations $\lam_1,\lam_2$. Denote by $A_{\ell_2}(\ell_1)$ the
collection of leaves from the forward orbit of $\ell_2$ linked with
$\ell_1$ together with $\ell_1$.  We call $A_{\ell_2}(\ell_1)$ the
\textit{accordion of $\ell_1$ with respect to $\ell_2$}. Abusing
notation, we often denote $A^+_{\ell_2}(\ell_1)$ 
by $A_{\ell_2}(\ell_1)$.
\end{dfn}

An important property of accordions is introduced in Definition~\ref{d:accord1}.

\begin{dfn}\label{d:accord1}
Say that $\ell_1$ has \textit{order preserving accordions} with
respect to $\ell_2$ if $A_{\ell_2}(\ell_1)\ne\{\ell_1\}$ and for all
$k\ge 0$, the restriction of $\si_d$ to
$A_{\ell_2}(\si_d^k(\ell_1))\cap \uc$ is order preserving. If
$\ell_1$ has order preserving accordions with respect to $\ell_2$
and vice versa, then these accordions are said to be
\textit{mutually order preserving} while $\ell_1,\ell_2$ are said to
have \textit{mutually order preserving accordions}.
\end{dfn}

Mutually order preserving accordions are not gaps of a single
sibling invariant geolamination as convex hulls of their forward
images may have intersecting interiors. Still, their dynamics
resembles that of gaps of sibling invariant geolaminations, and so
leaves from mutually order preserving accordions have a rigid
structure. This is supported by the findings made in Section 3 of
\cite{bopt14}; some of the main results from there are stated here.

\begin{lem}[Lemma 3.7 \cite{bopt14}]\label{l:sameperiod}
If $\ell_a=\ol{ab}$ and $\ell_x=\ol{xy}$, $a<x<b<y$, are linked
leaves with mutually order preserving accordions and $a, b$ are of
period $k$, then $x, y$ are also of period $k$.
\end{lem}

To state the main result of Section 3 of \cite{bopt14} we need
Definition~\ref{d:orderly}.

\begin{dfn}\label{d:orderly} Let $\ell_a=\ol{ab}$, $\ell_x=\ol{xy}$ be
linked chords. Set $B=\ch(\ell_a, \ell_x)$. Say that the
\emph{$\si_d$-dynamics of $B$ is orderly} if the following holds.

\begin{enumerate}

\item The order of vertices of $B$ is preserved under iterations of
    $\si_d$.

\item One of the following holds.

\begin{enumerate}

\item[(a)] All images of $B$ are pairwise disjoint.

\item[(b)] There exist the minimal $r, m$ and $n$ with
    $\si_d^r(B)\cap \si_d^{r+m}(B)\ne \0$ and
    $\si_d^{r+mn}(B)=B$, and if for any $i\ge 0$ we set
    $\si_d^{r+im}(a)=a_i, \si_d^{r+im}(b)=b_i,
    \si_d^{r+im}(x)=x_i, \si_d^{i+rm}(y)=y_i$, then either 
$$a_0<x_0<b_0\le a_1<y_0\le x_1<b_1\le a_2<y_1\le x_2<b_2\le \dots,$$
or 
$$x_0<a_0<y_0\le x_1<b_0\le a_1<y_1\le x_2<b_1\le a_2<y_2\le \dots,$$
and the set $\bigcup_{i=0}^{n-1}\si_d^{r+im}B$ is a component of the
orbit of $B$.

\item[(c)] The leaves $\ell_a$, $\ell_x$ are $($pre$)$periodic
    of the same eventual period of endpoints.

\end{enumerate}

\end{enumerate}

\end{dfn}

Despite its appearance, the order among vertices of $B$ and some of its
images given in Definition~\ref{d:orderly} can be easily described.

\begin{dfn}\label{d:posorder}
If there are chords $\ol{a_0b_0}, \ol{a_1b_1}, \dots, \ol{a_nb_n}$ with
$a_0\le a_1\le \dots \le a_n\le a_0$ and $b_0\le b_1\le \dots \le
b_n\le b_0$, we say that these chords are \emph{positively ordered} and
denote it by $\ol{a_0b_0}\le \ol{a_1b_1}\dots \le \ol{a_nb_n}\le
\ol{a_0b_0}$.
\end{dfn}

Here is how one can restate the long inequalities from
Definition~\ref{d:orderly}. Set $\si_d^{r+im}(\ell_a)=\ell^i_a,
\si_d^{r+im}(\ell_x)=\ell^i_x$. Then
$$\ell^0_a\le \ell^0_x\le \ell_a^1\le \ell_x^1\le \dots \ell^{n-1}_a\le
\ell_x^{n-1}\le \ell^n_a=\ell^0_a\le \ell^n_x=\ell^0_x$$ or
$$\ell^0_x\le \ell^0_a\le \ell_x^1\le \ell_a^1\le \dots \ell^{n-1}_x\le
\ell_a^{n-1}\le \ell^n_x=\ell^0_x\le \ell^n_a=\ell^0_a.$$ In
addition, only consecutive images of $\ell_a$ ($\ell_x$) in this
collection can intersect and at most at their endpoints.

The following is the main result of Section 3 of \cite{bopt14}

\begin{thm}[Theorem 3.12 \cite{bopt14}]\label{t:compgap}
Consider linked chords $\ell_a$, $\ell_x$ with mutually order
preserving accordions, and set $B=\ch(\ell_a,\ell_x)$. Then the
$\si_d$-dynamics of $B$ is orderly.
\end{thm}

In the rest of this subsection we assume that \emph{$\lam_1, \lam_2$
are linked cubic geolaminations}. Linked leaves of $\lam_1, \lam_2$
will often have images which are linked and have mutually order
preserving accordions. The only alternative is that they eventually
map to leaves which are not linked but share an endpoint. We call
this behavior \textit{collapsing (around a chain of spikes)}.

\begin{dfn}\label{d:accorder2}
Non-disjoint leaves $\ell_1\ne \ell_2$ are said to \emph{collapse
around chains of spikes} if there are two chains of spikes (one
chain in each of our geolaminations) connecting the same two
adjacent endpoints of $\ell_1, \ell_2$.
\end{dfn}


\begin{figure}
\includegraphics[width=4.5cm]{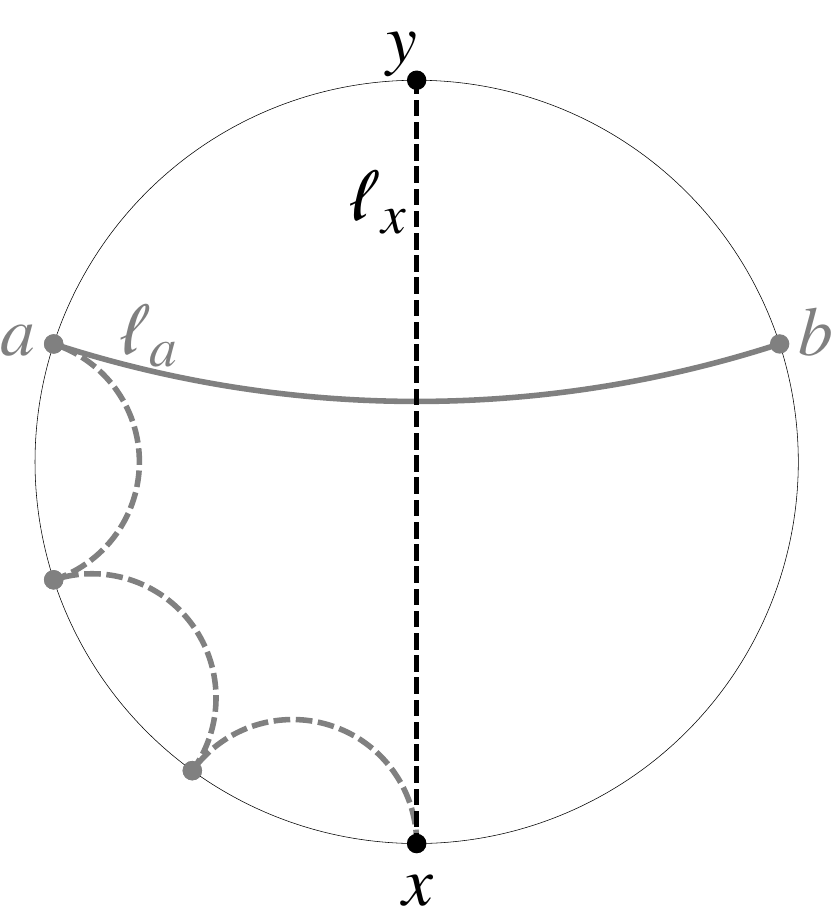}
{\caption{This figure illustrates Definition~\ref{d:accorder2}. Here
the leaves $\ell_a,\ell_x$ collapse around a chain of spikes shown
as dashed grey geodesics.}}\label{f:accorder2}
\end{figure}


We now list a few results of \cite{bopt14} on accordions of leaves
of $\lam_1 (\lam_2)$. Observe that the results obtained in
\cite{bopt14} hold in the degree $d$ case (i.e., for $\si_d$);
however, in this paper we confine ourselves to the case $d=3$ (in
particular, we define the linkage between geolaminations only for
cubic geolaminations) which is why we state results from
\cite{bopt14} in the cubic case only. If $X\subset \uc$ is a set of
points such that $\si_3|_X$ preserves orientation except that some
points may be mapped to one, we say that $\si_3|_X$ \emph{weakly
preserves orientation}.

\begin{lem}[Corollary 4.2 \cite{bopt14}]\label{l:linleabeh}
Let $\ell_1\in \lam_1, \ell_2\in\lam_2$; then for all $k\ge 0$, the
restriction of $\si_d$ to $A_{\ell_2}(\ell_1)\cap \uc$ is
(non-strictly) order preserving. Moreover, if $\ell_1$ and $\ell_2$
are linked leaves or share an endpoint then the order among their
endpoints is weakly preserved under $\si_3$
unless the common endpoint of $\ell_1, \ell_2$ is a common vertex of
associated critical \ql s of our geolaminations.
\end{lem}

Corollary 4.2 in \cite{bopt14} contains much more detail than
Lemma~\ref{l:linleabeh}. It serves as a major tool as it shows that
even though geolaminations $\lam_1, \lam_2$ are distinct, crossing
leaves of those geolaminations form a set which behaves more or less
like a gap of one geolamination.

\begin{lem}[Lemma 4.3 \cite{bopt14}]\label{l:linleabeh2}
Suppose that:

\begin{enumerate}

\item $\ell_1=\ol{a_1b_1}\in \lam_1, \ell_2=\ol{a_1b_2}\in\lam_2$ are
linked leaves or share an endpoint (in the latter case the shared
endpoint is not a common vertex of two associated generalized \ql
s);

\item at least of these leaves is not critical;

\item the order among the endpoints of these leaves is weakly preserved
so that, say, $a_1\ne a_2$ while
$\si_3(a_1)=\si_3(a_2)$.

\end{enumerate}

Then there are chains of spikes of $\lam_1$ and of $\lam_2$
connecting $a_1$ and $a_2$ with endpoints in the arc with endpoints
$a_1, a_2$ not containing points $b_1, b_2$.
\end{lem}

The main problem with applications of Lemma~\ref{l:linleabeh} is
that the order among the endpoints of linked leaves of
geolaminations $\lam_1, \lam_2$ is only weakly preserved. In other words, endpoints of these leaves
may collide. Lemma~\ref{l:linleabeh2} describes how this can happen.
and shows that if two linked leaves $\ell_1, \ell_2$ of linked
geolaminations have colliding endpoints then these endpoints are
connected with specifically located chains of spikes. This allows us
to apply Corollary~\ref{c:sibling_loc} and prove
Proposition~\ref{p:clps_sib_loc} which shows that the colliding
endpoints are in fact endpoints of specifically located sibling
leaves of $\ell_1, \ell_2$ respectively.

\begin{prop}\label{p:clps_sib_loc}
Suppose that $\lam_1, \lam_2$ are linked geolaminations and the leaves
$\ol{ab}\in\lam_1$ and $\ol{xy}\in\lam_2$ are non-critical, linked, and
such that $a<x<b<y<a$ and $\si_3(a)=\si_3(y)$. Then there exists a leaf
$\ol{ax'}$ which is a sibling of $\ol{xy}$ such that $x'<a<x<y<x'$.
\end{prop}

\begin{proof} Consider a disjoint sibling collection $\mathcal A$
of the leaf $\ol{xy}$ (formed by leaves of  $\lam_2$). 
By Lemma~\ref{l:linleabeh2} there is a chain of spikes of $\lam_2$
from $a$ to $y$ with endpoints contained in $[y, a]$. Choose the
leaf $\ol{ax'}\in \mathcal A$; then by Corollary~\ref{c:sibling_loc}
$x'<a<x<y<x'$.
\end{proof}

Finally, we state here Lemma 4.5 from \cite{bopt14}.

\begin{lem}[Lemma 4.5 \cite{bopt14}]\label{l:linleabeh3}
If $\ell_1\in \lam_1, \ell_2\in\lam_2$ are linked leaves or share an
endpoint and $\si_d^t(\ell^1)$ and $\si_d^t(\ell^2)$ do not collapse
around chains of critical chords for any $t$, then there exists an
$N$ such that the $\si_d^N$-images of $\ell^1, \ell^2$ are linked
and have mutually order preserving accordions. Thus, if $\m_1\in
\lam_1, \m_2\in\lam_2$ are non-disjoint leaves then for any $n$ the
leaves $\si_3^n(\m_1), \si_d^n(\m_2)$ are non-disjoint.
\end{lem}



\subsection{Proper Geolaminations}\label{ss:pr_geolam}

Given a lamination $\sim$, the generated geolamination $\lam_\sim$
is formed by the edges of $\sim$-classes (see
Definition~\ref{d:geolam}); if $\sim$ is invariant, the family of
all such geolaminations is denoted $\Q$. While $\Q$-geolaminations
are natural objects, they are often difficult to work with. 
To overcome these difficulties we define a wider and often more
convenient class of geolaminations. This will also be helpful in
dealing with geolaminations in the parameter space.

\subsubsection{Basic definitions and
properties}\label{sss:bas-def-pro}

First we introduce an inverse process to that from
Definition~\ref{d:geolam}.

\begin{dfn}\label{d:lam-from-geo}
Suppose that a family $\mathcal F$ of chords in $\cdisk$ is given.
Define an equivalence relation $\sim_{\mathcal F}$ as follows:
$x\sim_{\mathcal F} y$ if and only if there exists a finite
concatenation of chords of $\mathcal F$ joining $x$ and $y$. We say
that $\mathcal F$ \emph{generates (or gives rise to)} the
equivalence relation $\sim_{\mathcal F}$.
\end{dfn}

Clearly, $\sim_{\mathcal F}$ \emph{is} an equivalence relation for
\emph{any} collection of chords $\mathcal F$ ($\mathcal F$ does not
have to be invariant or even closed). We are interested in families
$\mathcal F$ such that the equivalence relation $\sim_{\mathcal F}$
is a lamination with the specific properties from
Definition~\ref{d:lam} (in the non-invariant case) and, in addition,
from Definition~\ref{d:si-inv-lam} (in the invariant case). Observe
that there are invariant geolaminations not from $\Q$ which generate
invariant laminations. For example, inserting a leaf into the
interior of a finite gap $G$ of an $\Q$-geolamination and then
inserting leaves from its grand orbit into appropriate sets from the
grand orbit of $G$ does not change the underlying equivalence
relation. Whether this can be done depends on the dynamics of $G$
(e.g., it can always be done if $G$ is critical or wandering), but
the resulting geolamination is no longer in $\Q$.

\begin{dfn}\label{d:proper0}
A family of chords $\mathcal F$ such that the generated equivalence
relation $\sim_\lam$ is a lamination is said to be \emph{proper}.
Denote by $\pr$ the family of all \emph{invariant} proper
geolaminations and by $\pr_d$ the set of $\si_d$-invariant
geolaminations from $\pr$.
\end{dfn}

The following lemma can be found in \cite{bmov13}.

\begin{lem}\label{l:propinva}
Suppose that a $\si_d$-invariant geolamination $\lam$ is proper
(i.e. generates a lamination $\sim_\lam$). Then $\sim_\lam$ is a
$\si_d$-invariant lamination.
\end{lem}

There are useful criteria for an invariant geolamination to be
proper.

\begin{dfn}\label{d:proper} Two non-degenerate leaves with common endpoint
$v$ and the same non-degenerate image are said to form a
\textit{critical wedge} (the point $v$ is said to be its
\emph{vertex}). Call a non-degenerate leaf \textit{improper} if it
has exactly one periodic endpoint.
\end{dfn}

The dynamics of improper leaves is described in
Lemma~\ref{l:impleaf}.

\begin{lem}\label{l:impleaf}
Let $\ell=\ol{px}$ be an improper leaf with $p$ periodic of period
$k$ and $x$ non-periodic. Then it has an eventual image-leaf which
is either a critical leaf with a periodic endpoint or a leaf with
exactly one periodic endpoint which maps to a periodic leaf (in the
latter case this image-leaf and the appropriate image of it form a
critical wedge).
\end{lem}

\begin{proof}
Consider the leaves 
$\si_d^{kn}(\ell)$. Suppose that all points $p, x,$ $\si_d^k(x),$
$\si_d^{2k}(x),$ $\dots$ are distinct. Then by
Lemma~\ref{l:triod_ord} applied to $\si_d^k$, we may assume that
$p<x<\si_d^k(x)<\si_d^{2k}(x)<\dots$. Since $\si_d$ is expanding,
this is impossible. Hence there exists a minimal $N$ such that
either $\si_d^{Nk}(x)=p$ or (again by Lemma~\ref{l:triod_ord})
$\si_d^{Nk}(x)$ is a $\si_d^k$-fixed point distinct from $p$. In the
first case there exists $0\le i<k$ such that
$\si_d^{k(N-1)+i}(\ell)$ is a critical leaf with a periodic endpoint
and in the second case there exists $0\le i<k$ such that
$\si_d^{k(N-1)+i}(\ell)\cup \si_d^{k(N-1)+i+1}$ is a critical wedge
with a periodic edge.
\end{proof}

Lemma~\ref{l:perioproper} presents the above mentioned criteria for
an invariant geolamination to be proper.


\begin{lem}[cf \cite{bmov13}]\label{l:perioproper}
The following properties are equivalent.

\begin{enumerate}

\item  The $\si_d$-invariant geolamination $\lam$ is proper (i.e., $\sim_{\lam}$ is a lamination).

\item The $\si_d$-invariant geolamination $\lam$ has
no critical leaves (wedges) with a periodic endpoint (vertex).

\item The $\si_d$-invariant geolamination $\lam$ has no improper leaves.

\item The $\si_d$-invariant geolamination has neither a critical leaf with a
periodic endpoint nor a critical wedge with a periodic leaf.


\end{enumerate}

\end{lem}

In particular, concatenations of leaves of $\lam$ are finite.

\begin{proof}
Parts (1) - (3) of the lemma are equivalent by \cite{bmov13}.
Clearly, (3) implies (4). The opposite direction follows from
Lemma~\ref{l:impleaf}.
\end{proof}

Consider the difference between a proper geolamination $\lam$ and
the $\Q$-ge\-o\-la\-mi\-na\-tion $\lam_{\sim_\lam}$ generated by the
lamination $\sim_\lam$.

\begin{dfn}\label{d:differ}
Let $\lam$ be a proper geolamination. For a $\sim_\lam$-class $\g$,
denote by $\g_\lam$ the collection of all leaves of $\lam$
connecting points of $\g$.
\end{dfn}

Lemma~\ref{l:differ} easily follows from the definitions.

\begin{lem}\label{l:differ}
Let $\lam$ be a proper geolamination and let $\g_\lam$ be a
$\sim_\lam$-class. Then all leaves of $\lam$ intersecting $\g_\lam$
are disjoint from all other leaves of $\lam$; for any points $a,
b\in \g_\lam$ there is a finite chain of leaves from $\g_\lam$
connecting $a$ and $b$.
\end{lem}

For a $\si_d$-invariant geolamination, by Theorem~\ref{t:summary}
and Definition~\ref{d:gap-inv} for any gap $G$ the map
$\si_d|_{\bd(G)}$ is the composition of a monotone map and a
covering map of some degree. If $G$ is of period $k$, then the same
holds for $\si_d^k:\bd(G)\to \bd(G)$; the degree of the
corresponding map is called the \emph{degree of $G$}. A
concatenation of chords is \emph{trivial} if it consists of one
chord.

Consider a lamination $\sim$. By \cite{bl02} for any infinite gap
$G$ of $\sim$ the set $G\cap \uc$ is a Cantor set (thus, by
Definition~\ref{d:geolam1} $G$ is a Fatou gap). Recall, that
$\psi_G:\bd(G)\to \uc$ is the monotone map which collapses all edges
of $G$ to points (since $G\cap \uc$ is a Cantor set, there are no
non-trivial concatenations of edges of $G$). By \cite{bl02}, a
periodic infinite gap $G$ of $\sim$ of period $k$ can be of one of
two types. If $G$ is of degree 1 we associate to it its irrational
\emph{rotation number $\rho$}, and $\si_d^k|_{\bd(G)}$ is
semiconjugate by $\psi_G$ to the circle rotation by the angle
$\rho$. Such gaps are called \emph{Siegel} gaps. Now,
$\si_d^k|_{\bd(G)}$ can be of degree $r>1$. Then by \cite{bl02} the
map $\psi_G$ semiconjugates $\si_d^k|_{\bd(G)}$ and $\si_r|_\uc$.
Such gaps $G$ are called \emph{Fatou gaps of degree $r$}. Observe
finally that by \cite{bl02} any infinite gap $G$ of $\sim$ is
(pre)periodic.

Now, if $\lam$ is a proper geolamination, we can construct the
lamination $\sim_\lam$, the geolamination $\lam_{\sim_\lam}$, and
apply to $\sim_\lam$ and $\lam_{\sim_\lam}$ the results of
\cite{bl02} quoted in the previous paragraph together with
Lemma~\ref{l:differ}. This gives the following description of
infinite gaps of proper geolaminations.  A periodic gap of period
$k$ is said to be \emph{identity return (under $\si_d^k$)} if
$\si_d^k$ fixes its vertices.


\begin{lem}\label{l:fatgaps}
Any infinite gap $G$ of a proper geolamination $\lam$ is a
(pre)\-pe\-ri\-odic Fatou gap. If $G$ is a $k$-periodic gap of
$\lam$ of degree $r$, then the following holds.

\begin{enumerate}

\item If $r=1$, then the map $\psi_G$ semiconjugates $\si_d^k|_{\bd(G)}$
and the rotation of $\uc$ by an irrational number $\rho_G$.

\item If $r>1$, then the map $\psi_G$ semiconjugates $\si_d^k|_{\bd(G)}$
and $\si_r$. If a maximal concatenation $L$ of periodic leaves
on the
boundary of $G$ is non-trivial and $\si_d^{kn}(L)\cap L\ne\0$, then its convex hull $H$ is a
periodic identity return gap under $\si^{kn}$.

\end{enumerate}

\end{lem}

\begin{proof}
We only need to prove the last claim from part (2). Suppose that
$L=\ell_1\cup\dots\cup\ell_m$ is a maximal concatenation (in this
order) of periodic leaves on the boundary of $G$ and $m>1$. Then the
convex hull $H$ of $L$ can be obtained by connecting the appropriate
endpoints of $\ell_1$ and $\ell_m$ with a chord which we denote by
$\n$. If $kn$ is the period of all the leaves in $L$ then
$\si_d^{kn}(H)=H$ and $\si_d^k$ fixes all the vertices of $H$
(observe that by Lemma~\ref{l:sameperiod} all leaves in $L$ have the
same period). Suppose that there exists a number $s<k$ such that
$\si_d^s(H)=H$. Then leaves in $L$ are not fixed under $\si_d^s$ and
one of them will have to be mapped onto $\n$, a contradiction.
\end{proof}

\begin{dfn}\label{d:siegfat}
Periodic gaps of degree $r$ of proper geolaminations are said to be
either \emph{Siegel gaps} (if $r=1$) or \emph{Fatou gaps of degree
$r$} (if $r>1$). For any gap $G$ of period $k$, a $k$-periodic edge
of $G$ is said to be \emph{refixed}; $\si_d^k|_{\bd(G)}$ is called a
\emph{remap}.
\end{dfn}

Let $\np$ be the set of all critical leaves with non-periodic
endpoints.

\begin{lem}\label{l:crit-2}
Let $\lam$ be a proper geolamination with a critical leaf $D\in
\np$. Let $A$ be the convex hull of the $\sim_\lam$-class containing
the endpoints of $D$. Then $A$ is a leaf or a finite gap and the
following cases are possible.

\begin{enumerate}

\item The set $A$ is critical
such that $\si_3|_A$ is $3$-to-$1$. Moreover, either $A$ is an
all-critical triangle, or there is a critical gap or leaf\,
$C\subset A, C\ne D$ of $\lam$ such that $\si_3|_C$ is of degree
$2$.

\item Otherwise there is a critical set $C$ that is either an infinite gap or the
convex hull of a $\sim_\lam$-class such that $\si_3|_C$ is of degree
two and one of the following holds.

\begin{enumerate}

\item[(a)] $C\cap \uc$ is finite, $C$ is disjoint from $A$, and $\si_3|_C$ is
two-to-one.

\item[(b)] $C$ is a periodic Fatou gap.

\item[(c)] $C$ is a preperiodic infinite gap which eventually maps to a
periodic Siegel gap $U$ with $D$ being an edge of $U$.

\end{enumerate}

\end{enumerate}

\end{lem}

\begin{proof} Since $\sim_\lam$ is a lamination, then $A$ is
either a leaf (then it coincides with $D$), or a finite critical
gap. If $A$ is of degree 3 then it is the unique critical set of
$\sim_\lam$ because we are considering the cubic case. The rest of
case (1) is left to the reader (recall that the degree of a gap is
defined in Definition~\ref{d:degreeg}).

Now, suppose that $A$ is of degree 2 and denote the second critical
set of $\sim_\lam$ by $C$. If $C\cap \uc$ is finite then (a) holds.
Thus we may assume that $C$ is an infinite gap. If $C$ is periodic,
then by Lemma~\ref{l:fatgaps} (b) holds. If $C$ is preperiodic then
an eventual image of $C$ is a periodic gap $U$ of degree 1 (it
cannot be of degree 2 because we work with the cubic case and there
is not enough criticality). By Lemma~\ref{l:fatgaps} $U$ is a Siegel
gap. It is well-known that in any cycle of Siegel gaps there must be
a gap with at least one critical edge (see, e.g., Lemma 2.13
\cite{bopt14}). Hence we may assume that $U$ has $D$ as its critical
edge (under the assumptions $D$ is the unique critical leaf of
$\lam$ as otherwise $C$ would have been another critical edge of $U$
while we assume that $C$ is an infinite gap).
\end{proof}

Proper geolaminations described in case (2)(c) of
Lemma~\ref{l:crit-2} are said to be \emph{Siegel geolaminations of
capture type}.


\subsubsection{Finding proper sub-geolaminations}
\label{sss:finding}

By Theorem~\ref{t:summary} the space $\Lam_d$ of all
$\si_d$-invariant sibling geolaminations is compact. This allows one
to assign geolamination(s) to every polynomial approximated by
polynomials with locally connected Julia sets: if $P_i\to P$ and
$P_i$'s are polynomials with locally connected Julia sets and
geolaminations $\lam_{P_i}$, then any limit geolamination
$\lim_{i\to\infty}\lam_{P_i}$ can be associated to $P$. However,
while 
$\lam_{P_i}$'s are proper, 
their limit geolaminations might be not proper 
because the set of proper geolaminations $\pr_d\subset \Lam_d$ is
not closed. To overcome this difficulty, we will develop techniques
to associate a proper geolamination to many geolaminations of
$\Lam_3$. A fact which plays a central role in this procedure is
that improper leaves are isolated. Say that a leaf $\ell$ of a
geolamination $\lam$ is a \emph{limit} leaf if it is the limit of a
sequence of leaves of $\lam$ distinct from $\ell$ itself.

Recall (Definition~\ref{d:proper}) that a non-degenerate leaf is
\textit{improper} if it has exactly one periodic endpoint. Also,
given an arc $I\subset \uc$ set $|I|$ to be its length.

\begin{prop}\label{p:impr_iso}
Let $\lam\in \Lam_d$ be a $\si_d$-invariant sibling geolamination.
Then every improper leaf of $\lam$ is isolated in $\lam$.
\end{prop}

\begin{proof}
Let $\ell$ be an improper leaf of $\lam$. By Lemma~\ref{l:impleaf},
an eventual image $\si_d^m(\ell)=\ol{px}$ of $\ell$ is either (1) a
critical leaf with one periodic endpoint $p$ of period $k$ or (2) a
leaf with exactly one periodic endpoint of period $k$ whose
$\si_d^k$-image $\si^k(\si_d^m(\ell))$ is a $\si_d^k$-periodic leaf
($\si^m_d(\ell)$ and $\si_d^{m+k}(\ell)$ form a critical wedge for
$\si_d^k$). We show by way of contradiction that
$\si_d^m(\ell)=\ol{px}$ is isolated in either case.

(1) In this case $\si_d^k(x)=p$. Since $\si_d$ is locally order
preserving then for any chord $\ol{p'x'}$ close to but disjoint from
$\ol{px}$ its $\si_d^k$-image crosses $\ol{px}$. Hence $\ol{px}$ is
a limit of leaves which meet $\ol{px}$. Since $\si_d$ is
ex\!panding, then the $\si_d^k$-image of a chord $\ol{xp'}$ close to
$\ol{px}$ crosses itself. On the other hand, let $\ol{px'}$ be a
leaf of $\lam$, and assume (without loss of generality) that
$p<x'<x$ and that $x'$ is close to $x$. Set
$\si_d^k(\ol{px'})=\ol{py}$; clearly, $x<y<p$. Then by
Lemma~\ref{l:sibling_loc} there exists a sibling leaf $\ol{xz}$ of
$\ol{py}$ emanating from $x$ and such that $p<z<x$. Since $x'$ is
close to $x$ we may assume that $[y, p]=\si_d^k[x', x]$. Since
$\si_d^k$ is expanding, $[y, p]=d\cdot [x', x]>|[x', x]|$. On the
other hand, the fact that $\ol{py}$ and $\ol{xz}$ are siblings and
the fact that $[y, p]$ is a small arc (because $x'$ and $x$ are
close) imply that the $|[z, x]|\ge |[y, p]|$. Thus, $|[z, x]|>|[x',
x]|$. Hence, $\ol{zx}$ crosses $\ol{px'}$, a contradiction.

(2) Let $\si_d^m(\ell)=\ol{px},$ $\si_d^{m+k}(\ell)=\ol{px'},$
$\si_d^k(x)=x'<p<x$ (thus, $\ol{px}\cup \ol{px'}$ is a critical
wedge). Then by Corollary~\ref{c:nocrit} so there are at most
finitely many leaves separating $\ol{px}$ and $\ol{px'}$. This
implies that if the leaf $\ol{px}$ is a limit leaf then there exists
a sequence of leaves $\ol{qt}$ which converge to $\ol{px}$ and are
such that $p\le q<t\le x$. Since $\si_d$ is expanding, if $p<q$ then
the $\si_d^k$-image of $\ol{qt}$ crosses $\ol{px}$. Thus, $p=q$.
Then, by continuity infinitely many leaves $\si_d^k(\ol{qt})$
separate $\ol{px}$ and $\ol{px'}$, a contradiction with the above.
\end{proof}

Call a disjoint sibling collection \emph{improper} if it includes an
improper leaf.

\begin{lem}\label{l:nipe}
If a leaf $\ell$ is periodic or has non-preperiodic endpoints, then
no disjoint sibling collection of $\ell$ is improper.
\end{lem}

\begin{proof} The case when the endpoints of $\ell$ are
non-preperiodic is obvious. If $\ell=\ol{xy}$ is periodic and
$\ell'$ is a disjoint sibling of $\ell$ then both endpoints of
$\ell'$ are non-periodic as desired.
\end{proof}

We need the following definition.

\begin{dfn}\label{d:propersbl}
Let $\lam\in\Lam_d\sm \pr_d$. Let $\lam^p\subset \lam$ be the set of
proper leaves $\ell\in \lam$ such that they and all their (eventual)
non-critical images have disjoint sibling collections consisting of
proper leaves.
\end{dfn}

Arguments similar to those used in the proof of
Theorem~\ref{t:cleanup} allow us to find proper sub-geolaminations.

\begin{lem}\label{l:pr_cleaning}
Let $\lam\in\Lam_d\sm \pr_d$. Then $\lam^p$ is a proper
$\si_d$-invariant sibling geolamination containing all periodic
leaves of $\lam$, all leaves of $\lam$ with non-preperiodic
endpoints, and all critical leaves of $\ell$ with non-periodic
endpoints.
\end{lem}


\begin{proof}
Let $\ell\in \lam^p$. Then $\ell$ is proper. By definition,
$\si_d(\ell)\in \lam^p$. Moreover, by definition $\ell$ has a
disjoint sibling collection consisting of proper leaves all of which
also belong to $\lam^p$. Now consider pullbacks of $\ell$ and show
that among them we can choose a disjoint sibling collection
consisting of proper leaves. Indeed, the only way $\ell$ can have an
improper pullback $\ell'$ is when $\ell$ is periodic while $\ell'$
has exactly one periodic endpoint. To handle this case, notice that
$\ell$ must also have a purely periodic pullback (e.g., if $\ell$ is
of period $s$ we can always choose $\si_d^{s-1}(\ell)$ as such a
leaf). Choosing this pullback and its disjoint sibling collection we
see that this entire collection consists of proper leaves as
desired.

By definition, it remains to show that $\lam^p$ is closed. To this
end we need to check if all removed leaves are isolated. Indeed, a
leaf $\ell$ is removed if for some $i\ge 0$ all disjoint sibling
collections of $\si_d^i(\ell)$ include an improper leaf (which is
isolated by Proposition~\ref{p:impr_iso}). By Lemma~\ref{l:nonisol}
this implies that $\si_d^i(\ell)$ is isolated. Indeed, if
$\si_d^i(\ell)$ is non-isolated, then by Lemma~\ref{l:nonisol} it
has a disjoint sibling collection consisting of non-isolated leaves
which (by Proposition~\ref{p:impr_iso}) are all proper, a
contradiction. Hence $\si_d^i(\ell)$ is isolated as desired. The
last claim of the lemma follows from Lemma~\ref{l:nipe} and the
definition of $\lam^p$.
\end{proof}

Proposition~\ref{p:pr_per_insert} provides conditions whereby (the
grand orbit of) a periodic leaf can be added to a proper
geolamination $\lam$ to create another proper geolamination.

\begin{prop}\label{p:pr_per_insert}
Let $\lam$ be a proper geolamination and $L$ be a cycle of periodic
leaves which do not cross 
leaves of $\lam$. Then there exists a proper geolamination
$\widehat\lam$ containing $L\cup \lam$ such that
$\widehat\lam\sm\lam$ consists of leaves which either eventually map
to $L$ or are limits of such leaves.
\end{prop}

\begin{proof}
The proof is rather straightforward, so we only sketch it. To
construct $\hlam$, we need to pull back $L$ in a step by step
fashion so that on each step disjoint sibling collections are
formed. This is immediate if a leaf which is being pulled back is
not contained in a gap which itself is the image of a critical set.
Otherwise it suffices to choose the pullbacks of the leaf in
question inside of the appropriate critical sets so that again we
will be getting one or several disjoint sibling collections.
Repeating this countably many times we will make the first step in
the construction. It is easy to see that when we take the closure,
the resulting geolamination is a $\si_d$-invariant sibling
geolamination \cite{bmov13}. It follows from
Proposition~\ref{p:impr_iso} that $\hlam$ is a proper geolamination.
\end{proof}

\section{Main Results}\label{s:results}

Consider cubic (geo)laminations. We need the following definition.

\begin{dfn}\label{d:admis}
Two generalized critical quadrilaterals form an \emph{admissible
(cubic) qc-portrait} if (1) if a generalized critical quadrilateral
has a periodic vertex and a non-degenerate image then it must have a
periodic edge, and (2) the elements of the qc-portrait and all their
images intersect at most over a common edge or vertex. By
Theorem~\ref{t:pullback_lam}, the sets of the orbits of elements of
an admissible qc-portrait $\qcp$ can be pulled back to form a
$\si_3$-invariant geolamination; any such geolamination is denoted
by $\lam_\qcp$.
\end{dfn}

Recall that $\np$ is the set of critical leaves with non-periodic
endpoints.

\begin{dfn}\label{d:notat1}
Let $\prnp_3$ be the set of all proper geolaminations with a
critical leaf from $\np$ except for Siegel geolaminations of capture
type. Let $\qnp_3$ be the collection of all admissible qc-portraits
with a critical leaf from $\np$ as the second element. For $D\in
\np$, $\Ss_D$ is the collection of admissible qc-portraits with $D$
as the second element.
\end{dfn}

Recall that the sets $\prnp_3, \qnp_3$ and $\Ss_D$ are endowed with
the Hausdorff metric and the induced topology. From now on whenever
we talk about a privileged qc-portrait of a geolamination $\lam\in
\prnp_3$ (the notion of a privileged qc-portrait can be found in the
Introduction and is formally given in Definition~\ref{d:pripo}) we
always assume that its second element is a critical leaf with
non-periodic endpoints.

\begin{thmA}
Each $\lam\in\prnp_3$ has at least one and no more than finitely
many privileged qc-portraits. For every $\qcp\in\qnp_3$, there
exists 
$\lam\in\prnp_3$ such that $\qcp$ is privileged for $\lam$.
Moreover, $\Ss_D$ is compact.
\end{thmA}

It turns out that privileged portraits properly capture dynamics.

\begin{thmB}
Suppose that $\qcp_1=(Q_1, D),\qcp_2=(Q_2, D)$ are privileged
qc-portraits for geolaminations $\lam_1,\lam_2\in\prnp_3$ such that
$\si_3(Q_1)\cap \si_3(Q_2)\ne \0$. Then
$\sim_{\lam_1}=\sim_{\lam_2}$.
\end{thmB}

Observe that if $\qcp_1,\qcp_2$ are linked privileged qc-portraits
for geolaminations $\lam_1,\lam_2\in\prnp_3$ then there exists a
critical leaf $D\in \np$ such that $\qcp_1, \qcp_2\in \Ss_D$.

In contrast to the quadratic case, where there are no linked minors,
there may be linked qc-portraits in $\qnp_3$, but the laminations
specified by these linked qc-portraits are the same.  We use this
result to study $\Ss_D$, which is naturally identified with the set
of $\pr_3$-geolaminations containing $D$.  Each $(Q,D)\in\Ss_D$ is
tagged by the chord or point (the \emph{minor}) $\si_3(Q)$. Denote
by $\Pp_D$ the family of all such leaves $\si_3(Q)$. We will show
that $\Pp_D$ is an appropriate cubic analog of Thurston's $\qml$.

\begin{thmC}
The family $\Pp_D$ is proper, so that the generated equivalence
relation $\sim_{\Pp_D}=\sim_D$ is a lamination. Each $\sim_D$-class
corresponds to a unique cubic lamination $\lam$. Conversely, every
cubic lamination which is not of capture Siegel type and is such
that endpoints of $D$ are equivalent corresponds to a
$\sim_D$-equivalence class.
\end{thmC}

If $D=\ol{ab}$, then Theorem C equips the set $\lami_D$ of cubic
laminations $\approx$ such that $a \approx b$ and which are not of
Siegel capture type with the quotient topology of the unit circle.
Also, suppose that a lamination $\approx$ with $a\approx b$ is of
Siegel capture type. Then there is only one way this can happen.
Namely, $\approx$ must then have a periodic Siegel gap $U$ with $D$
being an edge of $U$ and a non-periodic pullback $V$ of $U$ which
maps forward in a two-to-one fashion.

Finally we study the $\sim_D$-classes. First we modify the classical
notion of the \emph{minor} of a geolamination in the quadratic case.

\begin{dfn}[Minor sets in the  quadratic case \cite{thu85}]\label{d:minor2}
For a quadratic  la\-mi\-nation $\sim$, let $C$ be the
critical set of $\sim$. If $C$ is finite, let $m_\sim=\si_2(C)$. If
$C$ is a Fatou gap of degree two, let $m_\sim$ be the $\si_2$-image
of the refixed edge of $C$. The set $m_\sim$ is called the
\emph{minor set} of $\sim$. \end{dfn}

Note that in the quadratic case minor sets of quadratic laminations
coincide with the convex hulls of $\sim_\qml$-classes. We want to
extend these ideas to the cubic case. In the cubic case, similar to
the quadratic case, minor sets can be introduced for all laminations
$\sim$ from $\lami_D$ as images of the first critical set of $\sim$
(if it is finite) or the image of the refixed edge of the periodic
Fatou gap $U$ of degree two of $\sim$ (if it exists). However in the
cubic case there is a new phenomenon which causes these collections
of minor sets taken for laminations $\sim$ from $\lami_D$ to be insufficient.
We overcome this difficulty by modifying Definition~\ref{d:minor2}
below.

\begin{dfn}[Minor sets in the cubic case] \label{d:minset1}
If $\sim\in \lami_D$, let $D_\sim$ be the $\sim$-class containing
the endpoints of $D$. Also, if $\sim$ has a $k$-periodic critical
Fatou gap $U$, of degree two, let $M_\sim$ be the $\sim$-class of
the unique edge of $U$ of period $k$. Let $C_\sim$ be either the
first critical set of $\sim$ (if it is different from $D_\sim$ and
finite), $M_\sim$ (if the first critical set of $\sim$ is a periodic
Fatou gap $U$ of degree two), or $\ch(D_\sim)$ (if $\sim$ has a
unique critical class $D_\sim$). Set $\si_3(C_\sim)=m_\sim$ and call
$m_\sim$ the \emph{minor set} of $\sim$.
\end{dfn}

Now we are ready to state Theorem D.

\begin{thmD}
Classes of $\sim_D$ coincide with the minor sets $m_\sim$ where
$\sim\in \lami_D$.
\end{thmD}




\section{Privileged QC-Portraits for Proper Geolaminations}\label{s:priv_port}



Let us recall the definition of a privileged portrait.

\begin{dfn}\label{d:pripo}
If $\lam\in \prnp_3$ has a critical leaf $D$ with non-periodic
endpoints then a qc-portrait $\qcp=(Q, D)$ is called
\emph{privileged for $\lam$} if and only if $Q\subset C$ where $C\ne
D$ is a critical set of $\lam$ and either $C$ is finite, or $C$ is a
periodic Fatou gap of degree two and period $k$ and $Q$ is a
collapsing \ql{} which is a convex hull of a (possibly degenerate)
edge $\ell$ of $C$ of period $k$ and another edge $\hell$ of $C$
such that $\si_3(\ell)=\si_3(\hell)$.
\end{dfn}

Observe that edges of the set $Q$ from Definition~\ref{d:pripo} are
not necessarily leaves of $\lam$. For example, if $Q$ is contained
in a periodic Fatou gap of degree two then improper edges of $Q$ are
not leaves of the proper geolamination $\lam$.

Lemma~\ref{l:qcp-for-lam} immediately follows from the definitions
and Lemma~\ref{l:crit-2} and is stated here without proof.

\begin{lem}\label{l:qcp-for-lam}
Each $\lam\in\prnp_3$ has at least one and no more than finitely
many privileged qc-portraits.
\end{lem}

In the rest of this section we show that every $\qcp\in\fqcp_3^{np}$
is a privileged qc-portrait of some $\lam\in\prnp_3$. Recall
Definition~\ref{d:admis}.

\begin{dfn}\label{d:admipo}
An \emph{admissible (cubic) qc-portrait} is an ordered pair of
generalized critical \ql s $(A, B)$ such that $A\ne B$, and,
moreover, $A, B$ and all their images intersect at most over a
common edge or vertex, and if $A$ or $B$ has a periodic vertex then
it either has a degenerate image or has a periodic edge.
\end{dfn}

It turns out that if $B\in \np$ then it suffices to make sure that
the first part of the definition of an admissible qc-portrait holds,
the second one then will automatically follow. To prove that, given
a qc-portrait $\qcp$ one needs to use Theorem~\ref{t:pullback_lam}
and construct corresponding (sibling) invariant geolaminations each
of which is denoted by $\lam_\qcp$.

\begin{prop}\label{p:per_edge}
Let $(Q, D)$ be a cubic qc-portrait with $D\in \np$ and such that
$Q\ne D$ and all their images intersect at most over a common edge
or vertex. Moreover, if $Q$ is an all-critical triangle viewed as a
generalized \ql{} then we assume that if $Q$ has a periodic vertex
this vertex is considered as an edge of $Q$. Then $(Q, D)$ is
admissible. Moreover, if $Q$ has no periodic vertices, then any
geolamination containing $\{Q,D\}$ is proper and $\qcp=(Q,D)$ is a
privileged qc-portrait for it.
\end{prop}

\begin{proof}
Let $\lam_\qcp$ denote a geolamination which contains $(Q, D)$. If
$\lam_\qcp$ is not proper then there is a periodic point $p$ and a
critical leaf/wedge $L$ with vertex at $p$. Since the endpoints of
$D$ are non-periodic, $p\notin D$. Thus if vertices of $Q$ are
non-periodic, then $p\notin Q$ which implies that such a set $L$
does not exist, $\lam_\qcp$ is proper, and by definition $\qcp$ is a
privileged qc-portrait for $\lam_\qcp$. This proves the second claim
of the lemma.

To prove the first claim, assume that $Q$ is a collapsing \ql{} with
a periodic vertex $p$. Clearly, two edges of $Q$ with periodic
vertex, say, $p$, form a critical wedge $L$. Moreover, it follows
that $\lam$ has no critical leaves with a periodic endpoint, and the
unique critical wedge of $\lam$ is $L$. Since by
Lemma~\ref{l:impleaf} $\lam$ has at least one critical wedge with a
periodic edge, the second claim of the lemma follows.
\end{proof}

Corollary~\ref{c:prop_orb} immediately follows.

\begin{cor}\label{c:prop_orb}
Let $\qcp=(Q, D)\in\qnp_3$. Then $\si_3^k(Q)$ is not an improper
leaf for all $k>0$.
\end{cor}

\begin{proof}
If $Q$ has a periodic vertex, then by Proposition~\ref{p:per_edge}
it has a periodic edge, and the result is immediate.  Otherwise, by
Proposition~\ref{p:per_edge} any pullback geolamination $\lam_\qcp$
is proper.  Since all $\si_3$-images of $Q$ are in $\lam_\qcp$, none
of them are improper.
\end{proof}

To prove that if $\qcp=(Q, D)\in \qnp_3$ then $\qcp$ is a privileged
qc-portrait of a proper cubic geolamination we use the following
strategy. First we use Theorem~\ref{t:pullback_lam} and construct a
pull-back geolamination $\lam_\qcp$. Then we clean $\lam_\qcp$ using
Lemma~\ref{l:pr_cleaning} and obtain a proper geolamination
$\lam^{pr}$. Now consider two cases. The case when sets from $\qcp$
have no periodic vertices is easier to handle. In this case by
Proposition~\ref{p:per_edge} any pullback geolamination $\lam_\qcp$
is proper. By definition $\lam_\qcp$ is not of Siegel capture type.
Since $Q$ and $D$ by construction remain critical sets of
$\lam_\qcp$, $\qcp$ is privileged for $\lam_\qcp$ by definition.

Consider now the case when $Q$ has a periodic vertex. By
Proposition~\ref{p:per_edge} then $Q$ is either a critical leaf, or
an all-critical triangle, or a collapsing \ql{} with a periodic
edge. To show that $(Q, D)$ is a privileged qc-portrait for some
proper cubic geolamination we need the following fact.

\begin{prop}[\cite{sch09, thu85}]\label{p:quad_refixed_major}
Let $\ol c$ be a $\si_2$-critical leaf with a periodic endpoint $p$
of period $k>1$. Then there exists a unique leaf $M_{\ol c}=\ol{px}$
and a Fatou gap $V_{\ol c}$ such that $V_{\ol c}$ is of period $k$
and $M_{\ol c}$ is its refixed leaf.
\end{prop}

Now, construct a proper geolamination $\lam^{pr}$ out of $\lam_\qcp$
using Definition~\ref{d:propersbl} and Lemma~\ref{l:pr_cleaning}. If
$\qcp$ is a privileged qc-portrait for $\lam^{pr}$, there is nothing
to prove. Suppose that $\qcp$ is not a privileged qc-portrait for
$\lam^{pr}$. By definition the only way it can happen is when the
set $Q$ is contained in a quadratic Fatou gap $U$ of $\lam^{pr}$,
and, moreover, the period of the periodic edge of $Q$ is greater
than the period of the gap $U$. In this case we apply the map
$\psi_U$ which sends $Q$ to a $\si_2$-critical leaf $\ol c$ with a
periodic endpoint, say, $p$ of period $k>1$ ($k>1$ \emph{exactly
because} $\qcp$ is not privileged for $\lam_\qcp$). Then we use
Proposition~\ref{p:quad_refixed_major}, find a $\si_2$-periodic
Fatou gap $V$ for which $p$ is an endpoint of a refixed edge, and
pull $V$ back to $U$ using the projection map $\psi_U$. This finally
produces a proper geolamination for which $\qcp$ is a privileged
qc-portrait.


\begin{thm}\label{t:priv_qcp}
Every $\qcp=(Q, D)\in\qnp_3$ is a privileged qc-portrait of a
geolamination from $\prnp_3$.
\end{thm}

\begin{proof}
We may assume that $Q$ has a periodic (possibly degenerate) edge
$\ell_Q$ of vertex period $N$ and its sibling-edge $\hell_Q$.
Consider the family $\mathcal{A}_\qcp$ of \emph{all} cubic
geolaminations containing $Q$ and $D$ (by
Theorem~\ref{t:pullback_lam} $\mathcal{A}_\qcp\ne \0$). Among
geolaminations from $\mathcal{A}_\qcp$, choose a geolamination
$\lam$ with a maximal family of non-degenerate periodic leaves of
periods at most $N$. Clearly, such a geolamination $\lam$ exists. By
Lemma~\ref{l:pr_cleaning} we can find a proper geolamination
$\lam^{pr}\subset \lam$ with $\ell_Q\in \lam^{pr}$.

By way of contradiction we may assume that $\qcp$ is not privileged
for $\lam^{pr}$. Then there must exist a periodic Fatou gap $U_Q$
which contains $Q$. On the other hand, by Lemma~\ref{l:pr_cleaning}
$D$ is a leaf of $\lam^{pr}$. It follows that $U_Q$ is a Fatou gap
such that $\si_3|_{\bd(U_Q)}$ is of degree 2. Consider the case when
$\ell_Q$ is non-degenerate. Since by Lemma~\ref{l:pr_cleaning}
$\ell_Q\in \lam^{pr}$ then $\ell_Q$ is an edge of $U_Q$. Let us show
that $\hell_Q\in \lam^{pr}$. Indeed, by Lemma~\ref{l:pr_cleaning} we
include in $\lam^{pr}$ those leaves of $\lam$ which have disjoint
sibling collections consisting of proper leaves and whose all images
have the same property. Now, take $\hell_Q$ and its disjoint sibling
collection (which must exist by definition of a sibling invariant
geolamination). It follows that such collection \emph{must} include
$\ell_Q$. Hence $\hell_Q\in \lam^{pr}$ as desired. By construction,
$\ell_Q$ and $\hell_Q$ are edges of $U_Q$. Now, if $\ell_Q$ is
degenerate then $Q$ can be a critical leaf or an all-critical
triangle (viewed as a generalized \ql). In either case $\ell_Q$ is a
vertex of (an edge of) $U_Q$ and $\hell_Q$ is either a vertex of
$U_Q$ or an edge of $U_Q$ (the latter holds if $Q$ is an
all-critical triangle).


Let us show that then $\lam^{pr}$ cannot be a Siegel geolamination
of capture type. Indeed, suppose otherwise. Then $U_Q$ eventually
maps onto a periodic Siegel gap, and so the original set $Q$ has
edges with non-preperiodic leaves, a contradiction. Hence
$\lam^{pr}$ is not a Siegel geolamination of capture type.


The map $\phi_{U_Q}:\bd(U_Q)\to\uc$ collapses concatenations of
edges in $\bd(U_Q)$ to points and semiconjugates the remap of
$\bd(U_Q)$ to $\si_2$. Clearly, $\phi_{U_Q}(Q)=\ol{ab}$ is a
$\si_2$-critical leaf with a periodic endpoint, say, $z$. If $z=0$
then $\ell_Q$ is a refixed edge of $U_Q$ and we are done. Assume
that $z\ne 0$ and bring this to a contradiction. Set $M=M_{\ol{ab}}$
as in Proposition~\ref{p:quad_refixed_major}, and use
$\phi_{U_Q}^{-1}$ to pull the Fatou gap $V=V_{\ol{ab}}$ from
Proposition~\ref{p:quad_refixed_major} back to a gap $V_Q\subset
U_Q$ by taking full $\psi_{U_Q}$-preimages of points from $V\cap
\uc$ and then taking the convex hull of such preimages. Then
$\ell_Q$ is an edge of $V_Q$ and a periodic edge $M_Q$ of $V_Q$ of
vertex period $N$ maps to $M$ under $\psi_U$. By construction,
$Q\subset V_Q$ while $M_Q$ is a periodic leaf whose orbit consists
of leaves which do not cross $Q$. This contradicts the choice of the
geolamination $\lam$ as a geolamination with the maximal number of
non-degenerate periodic leaves of period at most $N$.
\end{proof}

\noindent\emph{Proof of Theorem A}. By Lemma~\ref{l:qcp-for-lam}
there is at least one and at most finitely many qc-portraits
privileged for a given geolamination from $\prnp_3$. On the other
hand, by Theorem~\ref{t:priv_qcp} every qc-portrait $\qcp\in\qnp_3$
is a privileged qc-portrait of a geolamination from $\prnp_3$.

Recall that for $D\in \np$, $\Ss_D$ is the collection of admissible
qc-portraits with $D$ as the second element. To show that $\Ss_D$ is
compact, let $(Q_n, D)\to (Q, D)$. By definition the fact that
$Q_n$'s are generalized critical \ql s implies that $Q$ is a
generalized critical \ql. Moreover, if $(Q, D)$ is not admissible
then either there exist numbers $i, j$ with $\si^i_3(Q)$ crossing
$\si^j_3(Q)$, or there exists a number $k$ such that $\si^k_3(Q)$
crosses $D$. In either case, the same crossing will have to take
place for the corresponding images of $Q_n$'s and $D$ contradicting
the fact that $(Q_n, D)$ is admissible. By
Proposition~\ref{p:per_edge} $(Q, D)$ is admissible. \hfill \qed

Notice that if $\qcp$ is a privileged qc-portrait of a proper
geolamination $\lam$ and all vertices of sets from $\qcp$ are
non-periodic then $\qcp$ remains a privileged qc-portrait of the
$\Q$-geolamination $\lam_{\sim_\lam}$ generated by $\lam$.

\section{Intersecting Minors of {$\qnp_3$}}\label{s:interminor}

By definition a proper family of chords $\mathcal F$ generates a
lamination $\sim_\mathcal F$ and, if $\mathcal F=\lam$ is also a
($\si_d$-)invariant geolamination, then $\sim_\lam$ is a
($\si_d$-)invariant lamination. Throughout the rest of this section
the following holds.

\begin{staa}
We fix a critical leaf $D=\ol{ab}\in \np$ with a positive arc $(a,
b)\subset\uc$ of length $\frac13$, set $\si_3(D)=d$, and denote by
$\De$ the all-critical triangle $\ch(a, b, v)$. Fix admissible
qc-portraits $(Q_1, D)$ $=$ $\qcp_1$ and $(Q_2, D)=\qcp_2$ such that
their \emph{minors} $\si_3(Q_1)$ and $\si_3(Q_2)$ are non-disjoint.
Assume that $\qcp_i$ is a privileged qc-portrait for a proper
geolamination $\lam_i$ for $i=1, 2$ such that neither $\lam_1$ nor
$\lam_2$ is of Siegel capture type. Set $\sim_{\lam_1}=\sim_1,
\sim_{\lam_2}=\sim_2$. Let $\hlam_1=\lam_{\sim_1}$ and $
\hlam_2=\lam_{\sim_2}$.
\end{staa}

By definition, $Q_1\ne D, Q_2\ne D$. On the other hand, it is
possible that $Q_1=\De$ or $Q_2=\De$.

Our aim is to prove Theorem B, i.e. to prove that $\sim_1=\sim_2$.
If $Q_1$ and $Q_2$ are strongly linked or share a spike, tools from
\cite{bopt14} apply and eventually imply the desired. However we
need to prove Theorem B under weaker assumptions that $\si_3(Q_1)$
and $\si_3(Q_2)$ are non-disjoint. Thus we need to study the case
when $\si_3(Q_1)$ and $\si_3(Q_2)$ are non-disjoint but $Q_1$ and
$Q_2$ are neither strongly linked nor share a spike.
Lemma~\ref{l:c1c2} shows that this is an exceptional case.

\begin{lem}\label{l:c1c2}
Qc-portraits $\qcp_1$ and $\qcp_2$ are linked unless one of the sets
$Q_1, Q_2$ coincides with $(a, x, v, x')$  while the other one
coincides with $(b, y, v, y')$.
\end{lem}

\begin{proof}
Let us show that if $\si_3(Q_1)\cap \si_3(Q_2)\ne \{d\}$ then
$\qcp_1$ and $\qcp_2$ are linked. Observe that $\si_3|_{[b, a]}$ is
two-to-one except that $d$ has three preimages $a, b, v\in [b, a]$.
Consider cases. Assume first that $\si_3(Q_1)\cap \si_3(Q_2)\cap
\disk\ne 0$. Since both minors have points inside $\disk$, they are
non-degenerate, and the generalized \ql s $Q_1, Q_2$ are true \ql s.
If the minors do not coincide, it immediately follows from
properties of $\si_3|_{[b, a]}$ and the assumptions that $Q_1$ and
$Q_2$ are strongly linked. If minors coincide and do not have $d$ as
an endpoint, it follows that $Q_1=Q_2$ (and hence by definition
$Q_1$ and $Q_2$ are strongly linked). Assume that
$\si_3(Q_1)=\si_3(Q_2)=\ol{dx}$ and denote by $x', x''\in (b, a)$
two points with $\si_3(x')=\si_3(x'')=x$. Then either $Q_1=Q_2$, or
otherwise we may assume that $Q_1=(a, x', v, x'')$ and $Q_2=\ch(b,
x', v, x'')$. Clearly, in this case again $Q_1$ and $Q_2$ are
strongly linked. Assume now that $\si_3(Q_1)\cap \si_3(Q_2)=\{y\}$
where $y\ne d$. Let $y', y''\in (b, a)$ so that
$\si_3(y')=\si_3('')=y$. Then both $Q_1$ and $Q_2$ share the spike
$\ol{y'y''}$. By definition, in all these cases $\qcp_1$ and
$\qcp_2$ are linked. If now $\si_3(Q_1)\cap \si_3(Q_2)=\{d\}$ then
both $Q_1$ and $Q_2$ have either $\ol{av}$ or $\ol{bv}$ as a spike.
If either $Q_1$ or $Q_2$ is $\De$, $\lam_1$ and $\lam_2$ are linked.
Hence one of the sets $Q_1, Q_2$ coincides with $(a, x, v, x')$
while the other one coincides with $(b, y, v, y')$.
\end{proof}

For brevity, if $\qcp_1$ and $\qcp_2$ are not linked we will say
that \emph{Case V} holds; without loss of generality we will then
always assume that $Q_1=(a, x, v, x')$ and $Q_2=(b, y, v, y')$. We
use Lemma~\ref{l:c1c2} to handle Case V. It turns out that then
either Theorem B follows from known results, or specific
restructuring allows us to find geolaminations which \emph{are}
linked and generate the same laminations as the given
geolaminations. This show that it is enough to prove Theorem B in
the case when $\qcp_1$ and $\qcp_2$ \emph{are} linked.

\begin{lem}\label{l:v-non-perio}
Suppose that Case V holds and $v$ is not periodic. Then
$\sim_1=\sim_2$.
\end{lem}

\begin{proof}
Since $a, b$ and $v$ are not periodic, $Q_1$ and $Q_2$ cannot be
critical \ql s generated by periodic Fatou gaps. Hence $\sim_1$
($\sim_2$) has a unique finite critical class $\g_1$ ($\g_2$) on
which $\si_3$ is three-to-one. Clearly, $\{a, b, v\}$ is contained
in both $\g_1$ and $\g_2$ so that $d=\si_3(a)\in \si_3(\g_1)\cap
\si_3(\g_2)$. By \cite{sch04}, $\g_1=\g_2$. A standard pull-back
argument then shows that $\sim_1=\sim_2$.
\end{proof}

To complete considering Case V we now assume that $v$ \emph{is}
periodic. We need a few lemmas from \cite{chmmo} where they are
proven in the degree $d$ case (we restate them for the cubic case).

\begin{lem}[Lemma 1.5 \cite{chmmo}]\label{l:long}
Suppose that $\lam$ is a cubic geolamination. Then for any leaf
$\ell=\ol{ab}$ of $\lam$ there exists $i\ge 0$ such that both circle
arcs with endpoints $\si_3^i(a), \si_3^i(b)$ are of length at least
$\frac14$.
\end{lem}

Given a circle chord $\ell=\ol{xy}$ which is not a diameter, let
$||\ell||$ be the length of the smallest of the two arcs in $\uc$
with endpoints $x, y$. Let $\ell_0=\ol{ab}$ be a chord with
$0<||\ell_0||=b-a<\frac13$. Let $I_0=[b, a+\frac13]$; let $I_1$ be
$I_0$ rotated by $\frac13$ and $I_2$ be $I_0$ rotated by $\frac23$.
Define the \emph{central strip $\cs(\ell_0)$} of $\ell_0$ as the
convex hull of $I_0\cup I_1\cup I_3$.

\begin{lem}[Theorem 2.10 \cite{chmmo}]\label{l:csl}
Suppose that $\lam$ is a cubic geolamination, $\ell$ is a leaf of
$\lam$ such that $\frac14\le ||\ell||<\frac13$. Let $j>0$ be the
least number such that $\si^j_3(\ell)$ has an endpoint in
$\cs(\ell)$. Then $\si_3^j(\ell)$ has endpoints in two distinct
circle arcs on the boundary of\, $\cs(\ell)$.
\end{lem}

We use these tools to prove the next lemma.

\begin{lem}\label{l:vperiod}
Let a proper geolamination $\lam$ have a periodic critical Fatou gap
$U$ of degree two and a refixed edge $\ell$. Let $\ell'$ be the
sibling of $\ell$ contained in $\bd(U)$. If  $\ell'$ is contained in
a finite critical gap $G'$ of $\sim_{\lam}$ then $\ell$ is the
unique refixed edge of $U$ and $\frac14\le ||\ell||<\frac13 $.
Moreover, the entire orbit of $U$ except for $U$ itself is disjoint
from $\cs(\ell)$.
\end{lem}

A weaker version of the last claim of the lemma, which may be easier
to apply, is that the images of $\ell$ do not intersect the two
all-critical triangles one of whose vertices is an endpoint of
$\ell$.

\begin{proof}
Suppose that $\ell=\ol{pq}$. Then $p+\frac13, p+\frac23\in G'$.
Since $\ell$ does not cross edges or diagonals of $G'$, $||\ell||\le
\frac13$. Moreover, since $\ell$ is refixed, $||\ell||<\frac13$. Let
$T$ be the all-critical triangle with vertex $p$. Clearly, images of
$\ell$ or any other periodic edge of $U$ do not cross edges of $T$.
Consider now any refixed edge $\hell$ of $U$ and prove that
$||\hell||\ge \frac14$.

Indeed, if $||\hell||<\frac14$, choose $i$ with $||\si_3^i(\hell)||$
maximal. Clearly, $T\subset \cs(\si_3^i(\hell))$. By
Lemma~\ref{l:long} we have that $\frac14\le
||\si_3^i(\hell)||<\frac13$, the set $\cs(\si_3^i(\hell))$ is
well-defined, and $\si_3^i(\hell)\ne \hell$. Since $\hell$ is itself
an eventual image of $\si_3^i(\hell)$, $\si_3^i(\hell)$ at some
later moment enters $\cs(\si_3^i(\hell))$. However, by
Lemma~\ref{l:csl} the \emph{first time} $\si^i_3(\hell)$ enters
$\cs(\si_3^i(\hell))$, the corresponding image of $\si_3^i(\hell)$
will have its endpoints in distinct arcs of the boundary of
$\cs(\si_3^i(\hell))$, a contradiction with maximality of
$||\si_3^i(\hell)||$. Hence $||\hell||\ge \frac14$, and the circle
arcs on the boundary of $\cs(\hell)$ are of length at most
$\frac13-\frac14=\frac1{12}$. Hence $\ell$ is the unique refixed
edge of $U$. The last claim of the lemma follows from
Lemma~\ref{l:csl}.
\end{proof}

Now we can go back to the case of two proper geolaminations for
which Case V holds.

\begin{lem}\label{l:v-perio}
Suppose that Case V holds and $v$ is periodic. Then we can replace
the \ql{} $Q_2$ by a \ql{} $Q_2'$ so that the new qc-portrait
$(Q_2', D)$ is still privileged for $\lam_2$ and is linked with
$\qcp_1$.
\end{lem}

\begin{proof}
Since $v$ is periodic and $\qcp_2$ is privileged for $\lam_2$,
either $\ol{yv}$ or $\ol{vy'}$ is a refixed edge of a periodic Fatou
gap $U$ of degree two. Assume that $\ol{vy'}$ is a refixed edge of
$U$. Then the leaves $\ol{vy}, \ol{by'}$ are not leaves of $\lam_2$.

By Lemma~\ref{l:vperiod} $\ol{vy'}$ is actually the unique refixed
edge of $U$. Let us show that then no edge of $U$ crosses the
critical chord $\ol{va}$. Indeed, suppose otherwise. Clearly, the
only way an edge $\ell$ of $U$ can cross $\ol{va}$ is when it has
$b$ as an endpoint. By \cite{bmov13} this implies that $\lam$ has a
leaf $\ol{zv}$ with $b<z<v$ which is a sibling leaf of $\ell$. Since
$\lam_2$ is proper, the leaf $\ol{vz}$ must be periodic of the same
period as $\ell$. Since $\ell$ is the unique refixed edge of $U$,
the leaf $\ol{zv}$ cannot be an edge of $U$. This implies that for
geometric reasons there still must exist an edge of $U$ coming out
of $v$ and different from $\ol{vy}$. However, as above, this edge of
$U$ will then have to be refixed, again contradicting
Lemma~\ref{l:vperiod}.

Thus, we can remove $\ol{vy}$ and $\ol{y'b}$ (i.e., two sides of
$Q_2$ which are not edges of $U$) and replace them by the critical
leaf $\ol{va}$. By definition, the qc-portrait $(\ol{va}, D)$ is
privileged for $\lam_2$. On the other hand, by definition $(\ol{va},
D)$ is linked with $\qcp_1$. This completes the proof.
\end{proof}

Thus, to prove Theorem B it suffices to consider the case when $Q_1,
Q_2$ are linked as by Lemma~\ref{l:v-perio} Case V can be reduced to
this case. So, from now on we may assume that $Q_1, Q_2$ are linked
and therefore $(\lam_1, Q_1)$ and $(\lam_2, Q_2)$ are linked.

We will need Lemma 4.8 from \cite{bmov13}.

\begin{lem}[Lemma 4.8 \cite{bmov13}]\label{l:4.8}
If $\lam$ is a $\si_d$-invariant geolamination and $\ell\in \lam$ is
a leaf such that $\si_d^n(\ell)$ is concatenated with $\ell$ for
some $n$ then the endpoints of $\ell$ are (pre)periodic.
\end{lem}

Proposition~\ref{p:per_de} studies possible intersections between
various leaves and edges of $\De$.

\begin{prop}\label{p:per_de}
If $\qcp=(Q,D)\in\Ss_D$ with $d\in \si_3(Q)$, then no eventual
$\si_3$-image of $Q$ crosses an edge of $\De$. If, moreover, $\lam$
is a proper geolamination for which $\qcp$ is privileged, then no
periodic leaf of $\lam$ crosses an edge of $\De$.
\end{prop}

\begin{proof} Let $\ol{av}$ be a spike of $Q$.
To prove the first claim of the lemma, assume that $Q=(a,
x, v, x')$ is a true \ql, and show that no eventual
$\si_3$-image of $Q$ crosses $\ol{bv}$. Indeed, suppose that
$\si_3^k(Q)$ crosses $\ol{bv}$. Then, since images of $Q$ cannot
cross $D$ or $Q$, it follows that the leaf $\si_3^k(Q)$ equals
$\ol{ay}$ and is concatenated with the leaf $\ol{xa}$. By
Lemma~\ref{l:4.8}, points $a$ and $y$ are (pre)periodic. By the
assumptions $a$ is not periodic. By Corollary~\ref{c:prop_orb},
$\ol{ay}$ and all its images are proper leaves. However if we now
follow the concatenation of leaves $\ol{ay}, \si_3^k(\ol{ay})$ etc,
we will at some point encounter the situation when exactly one
endpoint of some image of $\ol{ay}$ is periodic, a contradiction.

Consider the above situation for the sake of definiteness.
Assume that $\qcp$ is a privileged qc-portrait of a cubic proper
geolamination $\lam$; then leaves of $\lam$ cannot cross $\ol{av}$
or $D$. Hence if a leaf $\ell$ of $\lam$ crosses an edge of $\De$ then
$\ell=\ol{az}$ which cannot be periodic because $a$ is not
periodic.
\end{proof}

For a point $z\in\uc$, let $F(z)$ be the set of endpoints of leaves
from both $\lam_1$ and $ \lam_2$ containing $z$.  Note that $z\in
F(z)$. Recall that by Lemma~\ref{l:perioproper}(3) if an endpoint of
a leaf $\ell$ of a proper geolamination is periodic, the other
endpoint of $\ell$ is also periodic of the same period. Thus, the
notion of vertex period is well-defined for periodic leaves.


\begin{lem}\label{l:perfanorder}
If $z\in\uc$ is periodic under $\si_3$, then $\si_3|_{F(z)}$ is
order preserving.
\end{lem}

\begin{proof}
By Corollary~\ref{c:sameperiod} and since $\lam_1, \lam_2$ are
proper, all points in $F(z)$ are periodic of the same period. It
suffices to show that for any distinct points $x, y\ne z$ of
$F(z)$ the points $x, y, z$ are kept in order by $\si_3$. 
By Lemma~\ref{l:triod_ord} we may assume that $\ol{xz}\in\lam_1$ and
$\ol{yz}\in\lam_2$. 
We show that \emph{there is only one point such that when $z$
coincides with this point it becomes possible for the order among
the points $x, y, z$ to be reversed under $\si_3$}.

First let $Q_1, Q_2$ be neither strongly linked nor share a spike.
By Lemma \ref{l:c1c2} we may assume that $Q_1$ has a spike $\ol{av}$
and $Q_2$ has a spike $\ol{bv}$. If $z\ne v$ then $z\notin \De$. By
Proposition~\ref{p:per_de}, $\ol{xz}, \ol{yz}$ do not cross $\De$.
Hence $x, y, z$ belong either to $[a, b]$, or to $[b, v]$, or to
$[v, a]$. Also, $a, b\notin \{x, y, z\}$ because $a, b$ are
non-periodic. It follows that $\si_3|_{\{x, y, z\}}$ preserves the
order. Hence in this case the unique location for $z$ for which the
order among $a, y, z$ can be reversed is when $z=v$.

It remains to consider the case when $Q_1$ and $Q_2$ are strongly
linked or share a spike. In this case the qc-portraits $\qcp_1,
\qcp_2$ are linked and by Lemma~\ref{l:linleabeh} the order among
the points $x, y, z$ can be reversed only if $z$ is a common vertex
of associated \ql s. If there is a unique such periodic vertex, then
the claim is proven. Assume now that $Q_1$ and $Q_2$ have more than
one common \emph{periodic} vertex. If so, then both $Q_1$ and $Q_2$
must have at least three vertices (because for any critical chord
only one of its endpoints can be periodic). If, say, $Q_1$ is a
triangle then it must coincide with $\De$, hence in that case the
unique location of $z$ which allows for the reversal of orientation
is $z=v$, and the claim is once again verified.

Consider finally the case when both $Q_1$ and $Q_2$ are \ql s
sharing an edge $\ol{ut}$ with both endpoints periodic. Let $\g_i$
be the class of $\sim_i$ which contains $u$ and $t$. Since the class
$\g_i$ is periodic and finite, it is disjoint from $\ol{ab}$. Hence
the siblings of $u$ and $t$ in $(b,a)$ are unique and $Q_1=Q_2=Q$ is
the same quadrilateral. Let us show that then $\si_3|_{\{x,y,z\}}$
is order preserving. We may assume that $z=u$. Then there are two
cases. First, the points $x$ and $y$ may be separated by the spike
of $Q$ with an endpoint $u$. Then the remaining spike of $Q$ and $D$
cut $\cdisk$ in pieces one of which contains $x, y$ and $z$ which
implies that $\si_3|_{\{x,y,z\}}$ is order preserving. Otherwise the
points $x$ and $y$ are not separated by the spike of $Q$ with an
endpoint $u$. Then this very spike and $D$ again cut $\cdisk$ in
pieces one of which contains $x, y$ and $z$ which implies that
$\si_3|_{\{x,y,z\}}$ is order preserving. By induction
$\si^k_3|_{\{x,y,z\}}$ is order preserving for all $k$.

Thus, there is a unique location of $z$ for which the orientation of
the points $x, y, z$ can be reversed. However the periods of point
$x, y, z$ are the same, hence the power of $\si_3$ which maps $z$
back to itself sends both $x$ and $y$ back to themselves too. Since
by the above claim the reversal of orientation can happen along the
way only once, it follows that it does not take place at all. We
conclude that the claim holds for all periodic $z$.
\end{proof}

Suppose that say, $\ch(Q_1)=\De$. By definition either
$Q_1=(a,b,b,v)$, or $Q_1=(a,b,v,v)$, or $Q_1=(a,a,b,v)$. Then either
$\De$ is a gap of $\lam_2$ as well, or $Q_1$ and $Q_2$ are strongly
linked, or $Q_1$ and $Q_2$ share a spike. Let us show that either
$\ol{va}$ or $\ol{bv}$ is a spike of $Q_2$. This is obvious if $\De$
is a gap of $\lam_2$. Suppose that $\De$ is not a gap of $\lam_2$
and neither $\ol{va}$ nor $\ol{bv}$ is a spike of $Q_2$. Since by
definition $\ch(Q_2)\ne D$, then it is easy to see that no vertex of
$Q_2$ may coincide with $a, b$ or $v$. However then it follows that
$Q_2$ and $Q_1$ cannot be strongly linked. Thus, if $\ch(Q_1)=\De$
then we can always assume that either $\ol{va}$ or $\ol{bv}$ is a
spike of $Q_2$.

Proposition~\ref{p:no_per_collapse} studies vertex periods of linked
leaves of $\lam_1, \lam_2$.


\begin{prop}\label{p:no_per_collapse}
Suppose that a leaf $\ell_1\in\lam_1$ is periodic of vertex period
$n$ and a leaf $\ell_2\in\lam_2$ meets $\ell_1$. Then the order
among the endpoints of $\ell_1$ and $\ell_2$ is preserved under
$\si_3$ and $\ell_2$ is periodic with vertex period $n$.
\end{prop}

\begin{proof}
If $\ell_1$ and $\ell_2$ meet at an endpoint, the claim follows from
Lemma~\ref{l:perfanorder}. Otherwise we may assume that
$\ell_1=\ol{p_1q_1},\ell_2=\ol{p_2q_2}$ and $p_1<p_2<q_1<q_2$.
Observe that $\ell_2$ cannot be critical. Indeed, suppose otherwise.
Since $\ell_2$ and $\ell_1$ cross, $\ell_2\ne D$. Then there are the
following options for $Q_2$: $\ch(Q_2)=\De, \ch(Q_2)$ is the edge of
$\De$ not equal to $D$ or $\ell_2$, and finally $\ch(Q_2)=\ell_2$.
Let us first assume that $\ch(Q_2)=\ell_2=\ol{st}$. Then
$Q_2=(s,s,t,t)$. Since $\ell_1$ crosses $\ell_2$, $Q_1$ must be
located on one side of $\ell_1$. However, this would imply that
$Q_1$ and $Q_2$ can neither be strongly linked nor share a spike, a
contradiction.

Thus, either $\ch(Q_2)=\De,$ or $\ch(Q_2)$ is the edge of $\De$ not
equal to $D$ or $\ell_2$, and $\ell_2$ is either $\ol{bv}$ or
$\ol{va}$. Observe that in either case either $a, a$, or $a, b$, or
$b, b$ are consecutive vertices of $Q_2$. On the other hand, the
leaf $\ell_1$ crosses $\ell_2$ and, therefore, is linked with both
$\ol{bv}$ and $\ol{va}$. As before, it means that $Q_1$ is located
on one side of $\ell_1$ and hence, by definition, $Q_1$ cannot be
linked with $Q_2$, a contradiction. So we may assume that $\ell_2$
is not critical.

Let us show that the order among the
endpoints of $\ell_1, \ell_2$ is preserved. Indeed, suppose
otherwise. By Lemma~\ref{l:linleabeh}, without loss of generality we
may assume that $\si_3(p_1)=\si_3(q_2)$. By
Lemma~\ref{p:clps_sib_loc} there is a sibling $\ol{p_1p'_2}$ of
$\ol{p_2q_2}$ with $q_2<p'_2<p_1$. Since $\lam_2$ is proper, $p'_2$
is periodic. Then by Lemma~\ref{l:perfanorder},
$\si_3(p'_2)=\si_3(p_2)<\si_3(p_1)<\si_3(q_1)$. On the other hand,
since $(\lam_1, \qcp_1), (\lam_2, \qcp_2)$ are linked, then by
Lemma~\ref{l:linleabeh} the order among the points $p_1,p_2,q_1,q_2$
is weakly preserved under $\si_3$ so that
$\si_3(p_1)\le\si_3(p_2)\le\si_3(q_1)$, a contradiction. Thus, the
order among the endpoints of $\ell_1, \ell_2$ is preserved. In
particular, $\ell_2$ is not (pre)critical.

Let us show that $\ell_2$ is (pre)periodic. Consider
$A_{\ell_2}(\ell_1)=A$. By the above, the order of endpoints of
$\ell_1, \ell_2$ is preserved under $\si_3^{nk}$ for every $k\ge 0$.
Hence the sequence of points of intersection of leaves
$\si_3^{nk}(\ell_2)$ with $\ell_1$ is monotone on $\ell_1$ while all
leaves $\si_3^{nk}(\ell_2)$ are pairwise unlinked. Then there exists
a leaf $\ell_\infty$ equal to the limit of leaves
$\si_3^{nk}(\ell_2)$. Since $\si_3$ is locally expanding, $\ell_2$
is (pre)periodic. Moreover, by Lemma~\ref{l:sameperiod} the vertex
period of the eventual periodic image of $\ell_2$ is $n$.

Let us now show that in fact $\ell_2$ is periodic itself. Suppose
otherwise. Then for some $i$, $\si_3^i(\ell_2)$ and
$\si_3^{n+i}(\ell_2)$ are disjoint siblings which are both linked
with $\si_3^i(\ell_1)$. Clearly, $D$ does not separate
$\si_3^i(\ell_2)$ and $\si_3^{n+i}(\ell_2)$ as otherwise $D$ and
$\ell_1$ would cross. Thus, $\si_3^i(\ell_2)$ and
$\si_3^{n+i}(\ell_2)$ are located in the closure $\ol{W}$ of the
component $W$ of $\cdisk \setminus D$ with boundary circle arc of
length $\frac23$. Moreover, since $\si_3^i(\ell_2)$ and
$\si_3^{n+i}(\ell_2)$ are disjoint siblings, both spikes of $Q_2$
must separate $\si_3^i(\ell_2)$ and $\si_3^{n+i}(\ell_2)$. Thus,
both spikes of $Q_2$ cross $\ell_1$. Since $Q_1$ and $Q_2$ are
strongly linked or share a spike, it follows that $Q_1$ has a spike
linked with $\ell_1$, a contradiction. Thus, $\ell_2$ is periodic.
\end{proof}

Observe that if $\lam$ is proper then all $\sim_\lam$-classes are
finite. Hence for any point $x\in \uc$ there are at most finitely
many leaves of $\lam$ containing $x$.


\begin{prop}\label{p:limunlink}
No leaf $\ell_1\in\lam_1$ (resp. $\ell_2\in\lam_2$) can intersect
infinitely many leaves of $\lam_2$ (resp. $\lam_1$). In particular,
no leaf $\ell_1\in\lam_1$ (resp. $\ell_2\in\lam_2$) is linked with a
limit leaf of $\lam_2$ (resp. $\lam_1$).
\end{prop}

\begin{proof}
If $\ell_1\in\lam_1$ intersects infinitely many leaves
$\{\ell_2^j\}_{1}^\infty$ of $\lam_2$ then by
Lemma~\ref{l:linleabeh} $\si_3^i(\ell_1)$ intersects
$\si_3^i$-images of all leaves $\{\ell_2^j\}_{1}^\infty$ for any
$i$. Since for any point $x\in \uc$ there are at most finitely many
leaves of $\lam$ containing $x$, $\si_3^i(\ell_1)$ is linked with
infinitely many leaves of $\lam_2$. If $\si_3^i(\ell_1)$ is periodic
for some $i$ this would contradict
Proposition~\ref{p:no_per_collapse}. Thus we may assume that the
orbits of the endpoints of $\ell_1$ are infinite. As there are only
finitely many chains of spikes of $\lam_1$ or $\lam_2$, for some
$i_0$ and any $i\ge i_0$ the leaf $\si_3^i(\ell_1)$ cannot collapse
around any chain of spikes. Since $\si_3^{i_0}(\ell_1)$ is linked
with infinitely many leaves of $\lam_2$, Lemma~\ref{l:linleabeh}
implies that $\si_3^{i_0}(\ell_1)$ and those leaves of $\lam_2$ have
mutually order preserving accordions. Since true \ql s cannot be
wandering under $\si_3$ \cite{kiw02}, by Theorem~\ref{t:compgap}
$\ell_1$ is (pre)periodic, a contradiction.
\end{proof}

Let us study Fatou gaps of our geolaminations. For a Fatou gap $U$
of $\lam_1$, let $\phi_U:\partial U\to \uc$ be the map collapsing
all leaves in $\bd(U)$ to points. For a leaf $\ell_2\in\lam_2$ with
$\ell_2\cap\bd (U)\ne 0$, let $\ell', \ell''$ be edges of $\bd (U)$
which intersect $\ell_2$ such that $\phi_U(\ell')\ne \phi_U(\ell'')$
if possible. If so, define
$\phi_U(\ell_2)=\ol{\phi_U(\ell')\phi_U(\ell'')}$ and define
$\phi_U(\ell_2)=\phi_U(\ell_2\cap\bd (U))$ otherwise. Observe that
the map $\phi_U$ is defined not only for periodic but also for
non-periodic Fatou gaps (recall that all Fatou gaps are
(pre)periodic).


\begin{lem}\label{l:fgapproj1}
If $U$ is a Fatou gap of $\lam_1$, then, for any leaf $\ell_2\in
\lam_2$ such that $\ell_2\cap\partial U\ne\0$, the set
$\phi_U(\ell_2)$ is degenerate.
\end{lem}

\begin{proof}
Assume that $\phi_U(\ell_2)$ is non-degenerate. Let us show that we
may assume that $U$ is periodic. Indeed, any Fatou gap eventually
maps to a periodic Fatou gap. Moreover, since $\lam_1$ is not of
capture Siegel type then there are no critical preperiodic Fatou
gaps (in our case, i.e. with $D$ being a part of $\lam_1$, a
critical preperiodic Fatou gap which maps to a periodic Fatou gap of
degree greater than one is impossible). Thus, if $W$ is a
non-periodic Fatou gap of $\lam_1$ and $\ol{h}$ and $\ol{k}$ are its
edges which do not belong to the same concatenation of edges of $W$
then their images are distinct. Hence, if $\phi_U(\ell_2)$ is not
degenerate, we can keep mapping $U$ forward until $U$ maps to a
periodic gap $W$ and consider intersections of images $\ell_2$ with
appropriate images of $U$. By Lemma~\ref{l:linleabeh} the
intersections of $\ell_2$ with edges of $U$ are preserved under
iteration of $\si_3$. Since along the way to $W$ all preperiodic
images of $U$ are non-critical, we conclude that the image of
$\ell_2$ which intersects $W$ gives rise to a non-degenerate
$\phi_W$-image. Thus, we can assume without loss of generality that
$U$ is periodic from the very beginning.

Now, for geolaminations of $\prnp_3$, all periodic Fatou gaps are of
degree 1 or 2. Assume that $U$ is of degree 2.  Then $Q_1$ has an
edge $M$ which is a refixed edge of $U$.  Denote the  sibling edge
of $M$ in  $U$ by $M'$. Thus, $\phi_U(Q_1)=\ol{0\frac12}$. Since
$\si_3$-images of $\ell_2$ do not cross one another, $\si_2$-images
of $\phi_U(\ell_2)$ do not cross one another either. Then there
exists $k$ such that either (1) $\si_2^k(\phi_U(\ell_2))$ crosses
$\ol{0\frac12}$, or (2) $\si_2^k(\phi_U(\ell_2))=\ol{0\frac12}$. If
(1) holds, then $\si_3^k(\ell_2)$ crosses two improper leaves of
$Q_1$, contradicting the fact that $Q_2$ is strongly linked with
$Q_1$ or shares a spike with $Q_1$. If (2) holds, $\si_3^k(\ell_2)$
meets one refixed edge of $U$ and one edge of $U$ which is not
periodic contradicting Proposition~\ref{p:no_per_collapse} (indeed,
if $\si_3^k(\ell_2)$ is not periodic it cannot meet any refixed edge
of $U$ while if $\si_3^k(\ell_2)$ is periodic it cannot meet
non-periodic edge of $U$). Finally, if $U$ is of degree 1 then
$\phi_U(\ell_2)$ will eventually cross itself under irrational
rotation induced on $\phi_U(\bd(U))=\uc$, a contradiction.
\end{proof}




Recall that $\hlam_1$ and $\hlam_2$ are canonical geolaminations
constructed for laminations $\sim_1, \sim_2$ in turn generated by
the given geolaminations $\lam_1, \lam_2$.

\medskip

\noindent\emph{Proof of Theorem B.} First we prove that Fatou gaps
of $\hlam_1$ and $\hlam_2$ are the same. A Fatou gap of either
geolamination $\hlam_i$ cannot have a non-trivial (consisting of
more than one) concatenation of its edges as it will have to be
completed with an edge separating it from the rest of the Fatou gap
in question.
Let $\widehat U_1$ be a periodic Fatou gap of $\hlam_1$. The
corresponding Fatou gap $U_1$ of $\lam_1$ is of the same degree as
$\widehat U_1$. If for a leaf $\hell_2\in \hlam_2$ the chord
$\phi_{\widehat U_1}(\hell_2)$ is non-degenerate, then there must
exist a leaf $\ell_2\in \lam_2$ with non-degenerate image
$\phi_{U_1}(\ell_2)$, contradicting Lemma~\ref{l:fgapproj1}. Hence
there exists an infinite gap $\widehat U_2$ of $\hlam_2$ containing
$\widehat U_1$. Similarly, there exists an infinite gap of $\hlam_1$
containing $\widehat U_2$. Thus, $\widehat U_1=\widehat U_2$.

We claim now that leaves of $\hlam_1$ and $\hlam_2$ coincide. Call a
leaf of $\hlam_i$ a \emph{limit} leaf if it is the limit of leaves
of $\hlam_i$. We claim the limit leaves of $\hlam_1$ and the limit
leaves of $\hlam_2$ form the same family of leaves. Indeed, let
$\hell_1\in \hlam_1$ be a limit leaf, and prove that then $\hell_1$
is a leaf of $\hlam_2$. Observe that $\hell_1$ must be a leaf of
$\lam_1$ too. By Proposition~\ref{p:limunlink} $\hell_1$ is not
linked with any leaf of $\lam_2$ as any such leaf of $\lam_2$ will
be crossed by infinitely many leaves of $\lam_1$. This easily
implies that $\hell_1$ in fact is not linked with leaves of
$\hlam_2$ either. Moreover, the same arguments show that no leaf of
$\hlam_2$ can share an endpoint with $\hell_1$ and be otherwise
located on the side of $\hell_1$ from which $\hell_1$ is approached
by infinitely many leaves of $\hlam_1$. Suppose that $\hell_1$ is
not a leaf of $\hlam_2$. Then by the above $\hell_1$ is contained
(except the endpoints) in a Fatou gap $V$ of $\hlam_2$, a
contradiction with the above.

On the other hand, suppose that $\hell_1\in \hlam_1$ is not a limit
leaf of $\hlam_1$. Then on at least one side a Fatou gap is attached
to $\hell_1$ which implies that $\hell_1$ is a leaf of $\hlam_2$ too.
Thus, $\hlam_1=\hlam_2$ as desired. \hfill\qed

\section{Unlinkage of QC-Portraits of  {$\qnp_3$}}
\label{s:unlink}

In what follows we always assume that $\qcp_1=(Q_1,D)$,
$\qcp_2=(Q_2,D)\in\qnp_3$ are \emph{linked} and \emph{distinct}
privileged portraits of proper geolaminations
$\lam_1,\lam_2\in\Ss_D$. Recall that $\De=\ch(a,b,v)$ is the
all-critical triangle with  edge $D=\ol{ab}$ where $(a, b)$ is a
circle arc of length $\frac13$. Observe that $(\De, D)$ is an
admissible qc-portrait.

\begin{lem}\label{l:De}
If $Q_1=\De$, then $Q_2$ is an edge of $\De$ distinct from $D$, or
$Q_2$ is a collapsing \ql{} whose one spike is an edge of $\De$ not
equal to $D$ and whose edges do not cross $D$.
\end{lem}

\begin{proof}
Suppose that $\De$ is represented as a generalized \ql{}
by assigning to it the following vertices in the positive direction:
$(a, b, b, v)$. Denoting the vertices of $Q_2$ by $x_1\le x_2\le x
_3\le x_4$ we may assume that

$$a\le x_1\le b\le x_2\le b\le x_3\le v\le x_4\le a$$

\noindent which implies $x_2=b$. Since $\ol{x_2x_4}$ is a critical
chord then either $a=x_4$ or $v=x_4$. Clearly, the argument can be
repeated in other cases too; it shows that $Q_2$ has a spike which
is an edge of $\De$. In the degenerate case $Q_2$ coincides with an
edge of $\De$; recall here, that by definition $Q_2\ne D$. If $Q_2$
is a triangle, then $Q_2=\De$, contradiction with the assumption
that $Q_1\ne Q_2$. Now, suppose that $Q_2$ is a true \ql. Then it
cannot have $D$ as a spike. Hence its spike which is an edge of
$\De$ must be either $\ol{av}$ or $\ol{bv}$. Clearly, the edges of
$Q_2$ cannot cross $D$.
\end{proof}

It will be  convenient to separate the following easy fact into a lemma.

\begin{lem}\label{l:sharever}
The sets $Q_1$ and $Q_2$ share at most one periodic vertex.
\end{lem}

\begin{proof}
Indeed, otherwise the two shared vertices are not endpoints of $D$
and are not the endpoints of a critical chord. Hence if they are
shared by $Q_1$ and $Q_2$ then $Q_1=Q_2$ contradicting the
assumptions.
\end{proof}

\subsection{Slices of {$\prnp_3$}}\label{ss:slice}

For a family of chords $\F$, the equivalence relation $\sim_{\F}$ is
defined as an equivalence relation on $\uc$ under which two points
of $\uc$ are declared to be \emph{equivalent} if and only if a
finite chain of leaves of $\F$ connects them; we call the
equivalence relation $\sim_{\F}$ \emph{generated} by $\F$. We claim
that each admissible qc-portrait $\qcp\in \Ss_D$ is associated with
a unique lamination denoted by $\sim_\qcp$. Indeed, by Theorem A
there is a proper geolamination $\lam_\qcp$ for which $\qcp$ is
privileged; by definition, $\lam_\qcp$ gives rise to an equivalence
relation $\sim_{\lam_\qcp}=\sim_\qcp$. By Theorem B, $\sim_\qcp$ is
unique for a given admissible qc-portrait $\qcp$ from $\Ss_D$.
Indeed, if there are two distinct geolaminations $\lam_\qcp$ and
$\qcp$ is privileged for both of them, these two equal qc-portraits
can be viewed as linked which implies that the laminations generated
by the two proper geolaminations in question are the same (so that
our notation $\sim_\qcp$ is appropriate). Set
$\hlam_\qcp=\lam_{\sim_\qcp}$.
Denote by $\prnp_{3, D}$ the family of all proper geolaminations
with a critical leaf $D$ except for geolaminations of Siegel capture
type.

Since $\qcp\in \Ss_D$, the endpoints $a$ and $b$ of $D$ are
equivalent under $\sim_\qcp$. Now, if $\sim$ is a cubic lamination
not of Siegel capture type, and $a$ and $b$ are $\sim$-equivalent,
we can find a qc-portrait $\qcp=(Q, D)$ privileged for $\lam_\sim$.
Thus, the family $\lami_D$ of laminations $\sim_\qcp$ obtained as
described above for admissible qc-portraits $\qcp\in \Ss_D$ is in
fact the family of all cubic laminations $\sim$ not of Siegel
capture type with $a \sim b$. Observe that if $\sim$ has a periodic
Fatou gap $U$ of degree two then the non-degenerate refixed edge of
$U$ is unique (if it exists) because the remap is of degree two and
canonical geolaminations associated with laminations cannot have
concatenations of edges on their boundaries. Also, if $\sim$ has a
unique critical set then it has to coincide with the $\sim$-class of
$a, b$ and is therefore finite.

Recall that every $\si_2$-invariant lamination $\sim$ has a unique
\emph{minor} set $m_\sim$ which is the convex hull of the image of a
$\sim$-class of maximal diameter. By Thurston \cite{thu85}, the minor
sets of quadratic invariant laminations are pairwise disjoint and form
a lamination of the unit circle. That is, minor sets of quadratic
laminations form classes of $\qml$. Domains of $\cdisk/\qml$ (i.e.,
components of $\cdisk/\qml\sm \uc/\qml$) 
come from critical sets of quadratic laminations $\sim$ which are
periodic Fatou gaps $U$ of degree two; each such domain will be called
the \emph{Main Cardioid of $U$}. In particular, $U$ can coincide with
the entire closed unit disk in which case we have the Main Cardioid of
$\cdisk$ (or the Main Cardioid of the Mandelbrot set as it is usually
called). In general $U$ can be of period $n$ in which case its Main
Cardioid can be described as follows. Take all \emph{rotational} sets
$G$ of period $n$ inside $U$. Each such set has an unique longest edge
denoted by $M_G$. This includes Siegel disks $S$ of period $n$ inside
$U$ in which cases the critical edge of $S$ is its longest edge. Then
the boundary of the Main Cardioid of $U$ is formed by sets $\si_2(M)$
taken over all rotational sets $G\subset U$ of period $n$.

Let us go back to the cubic case. Recall that $\lami_D$ is the set
of all $\si_3$-invariant laminations $\sim$ such that $a\sim b$
except for laminations of \emph{Siegel capture type}, and that
$\Ss_D$ is the family of all admissible qc-portraits $(Q, D)$ with
$D$ as the second element. Each $(Q,D)\in\Ss_D$ is tagged by the
chord or point (the \emph{minor}) $\si_3(Q)$, and $\Pp_D$ is the
family of all such chords $\si_3(Q)$.

By Theorem A, every admissible qc-portrait is a privileged qc-portrait
of a \emph{proper} geolamination. Typically, but not always, an
admissible qc-portrait is privileged for a geolamination canonically
associated with a cubic lamination from $\lami_D$ (to each class of a
lamination one associates its convex hull and then considers a
geolamination formed by the edges of these convex hulls). Indeed,
suppose that $\sim$ is a lamination from $\lami_D$ which has a periodic
critical Fatou gap $U$ of degree two and period $k$ with a refixed edge
$M$ which is also an edge of an \emph{identity return} triangle $T$,
i.e. of a periodic triangle $T$ of period $k$. 
Denote the other two edges of $T$ by $\m$ and $\n$. If we remove the
grand orbit of $M$ from $\hlam_\sim$, we will get a new proper
geolamination $\lam$ for which the gap $U$ becomes a new gap $V$
with refixed edges $\m$ and $\n$. The convex hulls of $\m$ ($\n$)
with their siblings are two admissible qc-portraits privileged for
$\lam$.

This shows that all edges of the triangle $\si_3(T)$ belong to $\Pp_D$.
However $\si_3(\m)$ or $\si_3(\n)$ are \emph{not} minors of any
geolaminations canonically associated to a lamination from $\lami_D$.
Call the triangle $T$ the \emph{major set} of $\sim$; if such $T$ does
not exist we call $M$ the \emph{major set} of $U$. In other words, the
major set of $\sim$ is the $\sim$-class of the refixed edge of $U$.

\begin{dfn}[Minor sets in the cubic case] \label{d:minset}
For $\sim\in \lami_D$ let $\g_D$ be the $\sim$-class of $a$;
moreover, if $\sim$ has a periodic Fatou gap $U$ of degree two, let
$M_\sim$ be the major set of $U$. Let $C_\sim$ be either the first
critical set of $\lam_\sim$ (if it is different from $\g_D$ and
finite), the major set of $U$ (if the first critical set of $\sim$
is a periodic Fatou gap $U$ of degree two), or $\ch(\g_D)$ (if
$\sim$ has a unique critical class $\g_D$). Set
$\si_3(C_\sim)=m_\sim$ and call $m_\sim$ the \emph{minor set} of
$\sim$.
\end{dfn}

We are ready to prove  Theorems C and D.

\begin{thm}\label{t:extendc}
The family of chords $\Pp_D$ is proper and gives rise to the
lamination $\sim_D$. Convex hulls of $\sim_D$-classes are minor sets
$m_\sim$ where $\sim$ belongs to $\lami_D$.
\end{thm}

\begin{proof} Since by Theorem A the set $\Ss_D$ is compact,
$\Pp_D$ is a closed family of chords (minors). Consider the
equivalence relation $\sim_D=\sim_{\Pp_D}$. Clearly, $\sim_D$ is
closed (because $\Pp_D$ is closed). Assume that a lamination
$\approx$ from $\lami_D$ is a lamination such that its first
critical set $C_\approx$ is a finite gap disjoint from $D$. Then we
can insert collapsing \ql s in $C_\approx$ in several ways and
associate to it the corresponding proper geolaminations. Images of
these collapsing \ql s are minors of the corresponding proper
geolaminations. Since these images are edges (and possibly some
diagonals) of $m_\approx$, in this case $m_\approx$ is a class of
equivalence of $\sim_D$.

Similar arguments show that the same holds if $\approx$ has a unique
critical set coinciding with the convex hull $\ch(\g_\approx)$ of
the $\approx$-class $\g_\approx$ of $a$. On the other hand, the
analysis before Definition~\ref{d:minset} shows that if the first
critical set $C_\approx$ is a periodic Fatou gap of degree two then
the corresponding class of equivalence of $\sim_D$ is the image
$m_\approx$ of the major set $M_\approx$ of $\approx$. This covers
all possibilities for the lamination $\approx$. Since each minor
from $\Pp_D$ can be associated with a proper geolamination and then
with a lamination from $\lami_D$, these cases cover the entire
$\Pp_D$ and complete the proof.
\end{proof}

\end{document}